\documentclass[11pt]{ip-journal}

\usepackage{todonotes}

\usepackage{amsfonts, amssymb, amscd}
\numberwithin{equation}{section}

\usepackage[symbol]{footmisc}

\usepackage{bm}
\usepackage{verbatim}
\usepackage{mathrsfs}
\usepackage{graphicx}
\usepackage{tikz-cd}
\usepackage{subcaption}
\usepackage{listings}
\usepackage{subfiles}
\usepackage[toc,page]{appendix}
\usepackage{mathtools}
\usepackage{comment}
\usepackage{enumerate}
\usepackage{enumitem}
\usepackage[all]{xy}

\usepackage{graphicx}
\graphicspath{{images/}}

\usepackage{appendix}
\usepackage{hyperref}
\hypersetup{
	colorlinks=true,
	citecolor=red,
	linkcolor=blue,
	filecolor=magenta,      
	urlcolor=red,
}
\newcommand{\Ca}{\mathcal{A}}
\newcommand{\Cd}{\mathcal{D}}
\newcommand{\Ci}{\mathcal{I}}
\newcommand{\Co}{\mathcal{O}}
\newcommand{\Cr}{\mathcal{R}}
\newcommand{\Cs}{\mathcal{S}}
\newcommand{\Cx}{\mathcal{X}}
\newcommand{\Cy}{\mathcal{Y}}
\newcommand{\Cn}{\mathcal{N}}

\newcommand{\Cb}{\mathcal{B}}
\newcommand{\Ce}{\mathcal{E}}
\newcommand{\Cw}{\mathcal{W}}
\newcommand{\Cv}{\mathcal{V}}
\newcommand{\Ct}{\mathcal{T}}
\newcommand{\Cz}{\mathcal{Z}}
\newcommand{\Cq}{\mathcal{Q}}
\newcommand{\Cg}{\mathcal{G}}

\newcommand{\Cc}{\mathcal{C}}
\newcommand{\Cp}{\mathcal{P}}
\newcommand{\Cu}{\mathcal{U}}
\newcommand{\Cf}{\mathcal{F}}

\newcommand{\Mm}{\bm{\mathcal{M}}}
\newcommand{\Nn}{\bm{\mathcal{N}}}

\newcommand{\M}{\mathbf{M}}
\newcommand{\N}{\mathbf{N}}

\newcommand{\Supp}{\mathrm{Supp}}
\newcommand{\vol}{\mathrm{vol}}

\newcommand{\mult}{\mathrm{mult}}

\newtheorem{thm}{Theorem}[section]

\newtheorem{cor}[thm]{Corollary}
\newtheorem{lem}[thm]{Lemma}
\newtheorem{prop}[thm]{Proposition}

\newtheorem{claim}[thm]{Claim}

\theoremstyle{definition}
\newtheorem{defn}[thm]{Definition}

\theoremstyle{definition}
\newtheorem{rem}[thm]{Remark}

\theoremstyle{definition}

\begin{document}
	
	\title{Boundedness of polarized log Calabi--Yau fibrations}
	\author{Junpeng Jiao}

	\address{Yau Mathematical Sciences Center, Tsinghua University, Hai Dian District, Beijing, China, 100084}
	\email{jiao$\_$jp@mail.tsinghua.edu.cn}

	\date{\today}

	\begin{abstract}
		In this paper, we investigate the boundedness of log pairs with log Calabi--Yau fibration structures. We prove that total spaces of log Calabi--Yau fibrations are bounded modulo crepant birational equivalence when the Iitaka volumes of log canonical divisors are bounded and general fibers are in a bounded family of polarized log Calabi--Yau pairs.
	\end{abstract}
	
	\maketitle
	\setcounter{tocdepth}{1}
	\tableofcontents
	
	\section{Introduction}
	We work over the field of complex numbers $\mathbb{C}$. 
	
	Despite extensive research into boundedness properties of canonically polarized varieties and Fano varieties, leads to recent advances \cite{HMX18}\cite{Bir19}\cite{Bir21a}, relatively little is known about Calabi--Yau varieties. In this paper, our focus is on the boundedness of varieties $X$ with a fibration structure $f:X\rightarrow Z$ and a $\mathbb{Q}$-divisor $\Delta$ on $X$, such that $K_X+\Delta\sim_{\mathbb{Q}, Z} 0$ and a general fiber $(X_g,\Delta_g)$ is a log Calabi--Yau pair. We mention that a log Calabi--Yau pair refers to an lc pair $(Y,D)$ such that $K_Y+D\sim_{\mathbb{Q}} 0$. 
	
	Such varieties arise naturally in the minimal model program. For instance, varieties of intermediate Kodaira dimension are predicted by the abundance conjecture to admit a minimal model $X'$, where a suitable positive power of $K_{X'}$ is base point free and defines a morphism $f: X'\rightarrow Z$, which is equal to the Iitaka fibration. The fibers of $f$ have numerically trivial canonical divisors, and the base $Z$ is naturally endowed with the structure of a log general type pair. 
	
	Recent work by Brikar has analyzed the case when the Iitaka fibration has fibers of Fano type in \cite{Bir18}, and the results in \cite{Li20} apply to the study of the base of fibrations of Fano type. The boundedness of varieties with an elliptic fibration is considered in \cite{Bir18}, \cite{BDCS20}, \cite{CDCHJ18}, and \cite{Fil20}.

	\begin{defn}\label{definition of good minimal models}
		Let $d$ be a positive integer, $\Ci\subset [0,1]\cap\mathbb{Q}$ a DCC set, and $v,u$ two positive rational numbers. We define $\Cg_{klt}(d,\Ci,v,u)$ to be the set of log pairs $(X,\Delta)$ such that
		\begin{itemize}
			\item $(X,\Delta)$ is a projective $d$-dimensional klt pair,
			\item $\mathrm{coeff}(\Delta)\subset \Ci$, 
			\item $X$ has a fibration $f:X\rightarrow Z$ to a projective normal variety,
			\item $K_X+\Delta\sim _{\mathbb{Q}}f^*H$ for an ample $\mathbb{Q}$-divisor $H$ on $Z$,
			\item $\mathrm{vol}(H)=v$, and
			\item there exists an integral divisor $A$ on $X$, such that for a general fiber $X_g$ of $f$, $A_g:=A|_{X_g}$ is ample, and $\vol(A_g)=u$.
		\end{itemize}
	\end{defn}
	\begin{thm}\label{main theorem 1}
		Let $d$ be a positive integer, $\Ci\subset [0,1]\cap\mathbb{Q}$ a DCC set, and $v,u$ two positive rational numbers. Then
		$$\Cg_{klt}(d,\Ci,v,u)$$
		is bounded modulo crepant birational equivalence.
	\end{thm}
	
	We say a variety $Y$ is a Calabi--Yau variety if it is smooth, projective, simply connected, $K_Y\sim 0$ and $H^i(Y,\mathcal{O}_Y)=0$ for $0<i<\mathrm{dim}Y$.
	\begin{thm}\label{boundedness of Calabi--Yau with polarized fibration structure}
		Let $d$ be a positive integer and $u$ a positive rational number. Then the set of projective varieties $Y$ such that
		\begin{itemize}
			\item $Y$ is a Calabi--Yau variety of dimension $d$,
			\item there exists a fibration $f: Y\rightarrow Z$, and
			\item there is an integral divisor $A$ on $Y$, such that for any general fiber $Y_g$ of $f$, $A_g:=A|_{Y_g}$ is ample, and $\vol(A_g)=u$,
		\end{itemize}
		is bounded modulo flops. 
	\end{thm}
	For definition of boundedness up to birational equivalence or flops, see Definition \ref{birationally boundedness}.
	
	The existence of the divisor $A$ is a natural condition in many cases. For example, if a general fiber $X_g$ is Fano, then we can choose $A$ to be $-K_X$, or if $f$ is an elliptic fibration with a rational multi-section of fixed degree, then we may define $A$ to be the closure of the rational multi-section. Therefore, we have the following two direct applications of Theorem \ref{main theorem 1} and Theorem \ref{boundedness of Calabi--Yau with polarized fibration structure}.
	
	\begin{cor}\label{cor 1}
		Let $d,k$ be two positive integers, $\Ci\subset [0,1]\cap\mathbb{Q}$ a DCC set, and $v,u$ two positive rational numbers. Suppose $(X,\Delta)$ is a log pair such that 
		\begin{itemize}
			\item $(X,\Delta)$ is a $d$-dimensional klt pair,
			\item $\mathrm{coeff}(\Delta)\subset \Ci$,
			\item $X$ has a fibration $f:X\rightarrow Z$ to a projective normal variety,
			\item $K_X+\Delta\sim _{\mathbb{Q}}f^*H$ for an ample $\mathbb{Q}$-divisor $H$ on $Z$, and
			\item $\mathrm{vol}(H)=v$.
		\end{itemize}
		Suppose either $-K_{X_g}$ is ample for a general fiber $X_g$, or $f$ is an elliptic fibration with a rational multi-section of degree $k$, then $(X,\Delta)$ is bounded modulo crepant birational equivalence.
	\end{cor}
	
	\begin{cor}  \label{cor 2}
		Let $d,k$ be two positive integers and $u$ a positive rational number. Then the set of projective varieties $Y$ such that
		\begin{itemize}
			\item $Y$ is a Calabi--Yau variety of dimension $d$,
			\item there exists an elliptic fibration $f: Y\rightarrow Z$, and
			\item $f$ has a rational multi-section of degree $k$,
		\end{itemize}
		is bounded modulo flops. 
	\end{cor}

	Corollary \ref{cor 1} can induce a birational boundedness version of \cite[Theorem 1.2]{Bir18} and \cite[Theorem 1.1]{Fil20}. Corollary \ref{cor 2} gives another proof of \cite[Theorem 1.2]{BDCS20}. 
	
	Note that Calabi--Yau varieties may have infinitely many models that are isomorphic in codimension one, so the boundedness modulo flops does not imply boundedness. A celebrated conjecture, the Kawamata--Morrison Conjecture, predicts that the isomorphism types of such models are just finitely many distinct ones.

	\noindent\textbf{Acknowledgement}. The author would like to thank his advisor Christopher D. Hacon for his encouragement and constant support. He would like to thank Chuanhao Wei, Jingjun Han, Jihao Liu, and Yupeng Wang for their helpful comments. The author was partially supported by NSF research grant no: DMS-1952522 and by a grant from the Simons Foundation; Award Number: 256202.

	\section{Preliminaries}
	
	\subsection{Conventions}
	We will use the notation in \cite{KM98} and \cite{Laz04}.
	
	Let $\Ci \subset \mathbb{R}$ be a subset, we say $\Ci$ satisfies the descending chain condition (DCC) if there is no strictly decreasing infinite subsequence in $\Ci$.
	
	A contraction is a projective morphism $f:X\rightarrow Z$ with $f_*\mathcal{O}_{X}=\mathcal{O}_Z$, hence it is surjective with connected fibers. A fibration means a contraction $X\rightarrow Z$ such that $\mathrm{dim}X>\mathrm{dim}Z$. A birational contraction is a birational map $\phi:X\dashrightarrow Z$ such that $\phi^{-1}$ does not contract any divisor.
	
	Throughout this paper, for a fibration $X\rightarrow Z$, we use $X_\eta$ to denote the generic fiber and $X_g$ to denote a general fiber. For a $\mathbb{Q}$-divisor $\Delta$ on $X$, we write $\Delta_\eta:=\Delta|_{X_\eta}$ and $\Delta_g:=\Delta|_{X_g}$.
	
	\subsection{Divisors}
	For a birational map $f: Y\dashrightarrow X$ and a divisor $B$ on $X$, $f_*^{-1}(B)$ denotes the strict transform of $B$ on $Y$, and $\mathrm{Exc}(f)$ denotes the sum of reduced exceptional divisors of $f$.
	
	For a $\mathbb{Q}$-divisor $D$, the map defined by the linear system $|D|$ means the map defined by $|\lfloor D\rfloor|$. Given two $\mathbb{Q}$-divisors $A,B$, $A\sim_{\mathbb{Q}} B$ means that there is an integer $m>0$ such that $m(A-B)\sim 0$. Let $D:=\sum_{i=1}^k a_i D_i$ be a $\mathbb{Q}$-divisor, where $\{D_1,\cdots, D_k\}$ are the irreducible components of $\Supp (D)$, we define $\mathrm{coeff}(\Delta)$ to be the set $\{a_1,\cdots, a_k\}$.
	
	Let $D$ be a $\mathbb{Q}$-divisor, we write $D=D_{> 0}-D_{< 0}$ as the difference of the effective part and negative part. Let $f:X\rightarrow Z$ be a fibration and $D$ a $\mathbb{Q}$-divisor on $X$, we write $D=D_v+D_h$, where $D_v$ is the $f$-vertical part and $D_h$ is the $f$-horizontal part. We say a closed subvariety $W\subset X$ is vertical over $Z$ if $f(W)\subsetneqq Z$, horizontal over $Z$ if $f(W)=Z$.
	
	Let $f:X\rightarrow Y$ be a fibration of normal varieties, $D$ a $\mathbb{Q}$-divisor on $X$. We say that $D$ is $f$-very exceptional if $D$ is $f$-vertical and for any prime divisor $P$ on $Y$ there is a prime divisor $Q$ on $X$ which is not a component of $D$ but $f(Q)=P$, i.e. over the generic point of $P$ we have: $\Supp (f^*P) \not\subset \Supp (D)$.

	We define the Iitaka volume of a $\mathbb{Q}$-divisor by analogy with the definition of the volume.
	Let $X$ be a normal projective variety and $D$ be a $\mathbb{Q}$-Cartier $\mathbb{Q}$-divisor. When the Iitaka dimension $\kappa(D)$ of $D$ is non-negative, the Iitaka volume of $D$ is defined to be
	\begin{equation}
		\mathrm{Ivol}(D):=\limsup_{m\rightarrow \infty}\frac{\kappa(D)! h^0(X,\Co_X(\lfloor mD\rfloor))}{m^{\kappa(D)}}.
	\end{equation}
	\subsection{Pairs}
	
	A sub-log pair $(X,\Delta)$ consists of a normal variety $X$ and a $\mathbb{Q}$-divisor $\Delta$ on $X$ such that $K_X+\Delta$ is $\mathbb{Q}$-Cartier. A sub-log pair $(X,\Delta)$ is called a log pair if $\Delta\geq 0$. If $g: Y\rightarrow X$ is a projective birational morphism and $E$ is a divisor on $Y$, the discrepancy $a(E,X,\Delta)$ is $-\mathrm{coeff}_{E}(\Delta_Y)$, where $K_Y+\Delta_Y :=g^*(K_X+\Delta) $. Let $a\in (0,1)$ be a rational number, a sub-log pair $(X,\Delta)$ is call sub-klt (resp. sub-lc, sub-$a$-lc) if for every projective birational morphism $Y\rightarrow X$ as above, $a(E,X,\Delta)>-1$ (resp. $\geq -1,\geq -1+a$) for every divisor $E$ on $Y$. A sub-log pair $(X,\Delta)$ is called klt (resp. lc, $a$-lc) if $(X,\Delta)$ is sub-klt (resp. sub-lc, sub-$a$-lc) and $(X,\Delta)$ is a log pair.

	Let $(Y,\Delta_Y)$ and $(X,\Delta)$ be two sub-log pairs and $h:Y\rightarrow X$ a projective birational morphism, we say $(Y,\Delta_Y)\rightarrow (X,\Delta)$ is a crepant birational morphism if $K_Y+\Delta_Y\sim_{\mathbb{Q}}h^*(K_X+\Delta)$, two sub-log pairs $(X_i,\Delta_i),i=1,2$ are crepant birationally equivalent if there is a sub-log pair $(Y,\Delta_Y)$ and two crepant birational morphisms $(Y,\Delta_Y)\rightarrow (X_i,\Delta_i),i=1,2$.
	
	A generalized log pair $(X,\Delta+M)$ consists of a normal projective variety $X$ equipped with a projective birational morphism $X'\xrightarrow{f} X$ where $X'$ is normal, a $\mathbb{Q}$-divisor $\Delta$, and a $\mathbb{Q}$-Cartier nef divisor $M'$ on $X'$ such that $K_{X}+\Delta+M$ is $\mathbb{Q}$-Cartier, where $M=f_*M'$. Let $\Delta'$ be the $\mathbb{Q}$-divisor such that $K_{X'}+\Delta'+M'=f^*(K_X+\Delta+M)$ and $a\in (0,1)$ a rational number, for a prime divisor $P$ over $X$, we define the generalized discrepancy of $P$ with respect to $(X,\Delta+M)$ as $a(P,X,\Delta+M):=a(P,X',\Delta')$ and we call $(X,\Delta+M)$ a generalized sub-klt (resp. generalized sub-lc, generalized sub-$a$-lc) pair if $(X',\Delta')$ is sub-klt (resp. sub-lc, sub-$a$-lc), a generalized klt (resp. generalized lc, generalized $a$-lc) pair if further $\Delta\geq 0$.

	Let $(X,\Delta)$ be a sub-lc pair (resp. $(X,\Delta+M)$ be a generalized sub-lc pair). A log canonical place of $(X,\Delta)$ (resp. a generalized log canonical place of $(X,\Delta+M)$) is a prime divisor $D$ over $X$ such that $a(D,X,\Delta)=-1$ (resp. $a(D,X,\Delta+M)=-1$). A log canonical center (resp. a generalized log canonical center) is the image on $X$ of a log canonical place (resp. a generalized log canonical place). Similarly, a log place of $(X,\Delta)$ (resp. a generalized log place of $(X,\Delta+M)$) is a prime divisor $D$ over $X$ such that $a(D,X,\Delta)\in [-1,0)$ (resp. $a(D,X,\Delta+M)\in [-1,0)$). A log center (resp. a generalized log center) is the image on $X$ of a log place (resp. a generalized log place).

	A log smooth pair is a log pair $(X,\Delta)$ where $X$ is smooth and $\Supp (\Delta)$ is a simple normal crossing divisor. Assume $(X,\Delta)$ is a log smooth pair and assume $\Supp (\Delta)=\sum_1^r\Delta_i$, where $\Delta_i$ are the irreducible components of $\Delta$. A stratum of $(X,\Delta)$ is an irreducible component of $\bigcap_{i\in I}\Delta_i$ for some $I\subset {1,\cdots,r}$. If $\mathrm{coeff}(\Delta)=1$, a stratum is a log canonical center of $(X,\Delta).$

	\begin{defn}\label{birationally boundedness}
		We say that a set $\mathscr{X}$ of log pairs is bounded modulo crepant birationally equivalence (respectively flops) if there is a projective morphism $\Cw\rightarrow \Ct$, where $\Ct$ is of finite type, and a $\mathbb{Q}$-divisor $\Cb$ on $\Cw$ such that for every $(X,\Delta)\in \mathscr{X}$, there is a closed point $t\in \Ct$ such that $(\Cw_t,\Cb_t)$ is crepant birationally equivalent (respectively crepant birationally equivalent and isomorphic in codimension one) to $(X,\Delta)$.
	\end{defn}

	\subsection{$\mathbf{b}$-divisors}
	Let $X$ be a variety, we say that a formal sum $\mathbf{B}=\sum a_\nu \nu$, where the sum ranges over all divisorial valuations of $X$, is a $\mathbf{b}$-divisor, if the set
	$$F_X=\{\nu\ |\ a_\nu \neq 0\text{ and the center $\nu$ on $X$ is a divisor}\},$$
	is finite. For any birational map $Y\dashrightarrow X$, the trace $\mathbf{B}_Y$ of $\mathbf{B}$ is the sum $\sum a_\nu B_\nu$, where the sum now ranges over the elements of $F_Y$.
	\begin{defn}\label{b-divisors}
		Let $(X,\Delta)$ be a log pair. If $\pi:Y\rightarrow X$ is a projective birational morphism, then we may write
		$$K_Y+\Delta_Y\sim_{\mathbb{Q}}\pi^*(K_X+\Delta)$$
		Define a $\mathbf{b}$-divisor $\mathbf{L}_\Delta$ by setting $\mathbf{L}_{\Delta,Y}=\Delta_{Y,> 0}$.
		
	\end{defn}
	\subsection{Minimal models}
	Let $\phi:X\dashrightarrow Y$ be a projective birational map of normal quasi-projective varieties so that $\phi^{-1}$ contracts no divisors. If $D$ is a $\mathbb{Q}$-Cartier divisor on $X$ such that $D':=\phi_*D$ is $\mathbb{Q}$-Cartier then we say that $\phi$ is $D$-non-positive (resp. $D$-negative) if for a common resolution $p:W\rightarrow X$ and $q:W\rightarrow Y$, we have $p^*D=q^*D'+E$ where $E\geq 0$ and $p_*E$ is $\phi$-exceptional (resp. $\Supp(p_*E)=\mathrm{Exc}(\phi)$). Suppose that $f: X\rightarrow S$ and $f^m: X^m\rightarrow S$ are projective morphisms, $\phi: X\dashrightarrow X^m$ is a projective birational morphism over $S$ and $(X,\Delta)$ and $(X^m,\Delta^m)$ are lc pairs, klt pairs or dlt pairs where $\Delta^m=\phi_*\Delta$. If $a(E,X,\Delta)>a(E,X^m,\Delta^m)$ for all prime $\phi$-exceptional divisors $E\subset X$, $X^m$ is $\mathbb{Q}$-factorial and $K_{X^m}+\Delta^m$ is nef over $S$, then we say that $\phi:X\dashrightarrow X^m$ is a minimal model over $S$. If instead $a(E,X,\Delta)\geq a(E,X^m,\Delta^m)$ for all divisors $E$ and $K_{X^m}+\Delta^m$ is nef, then we call $X^m$ a weak log canonical model of $K_X+\Delta$. A weak log canonical model $\phi: X\dashrightarrow X^m$ is called a semi-ample model if $K_{X^m}+\Delta^m$ is semi-ample, and it is called a good minimal model if $\phi$ is also a minimal model. Notice that by the negativity lemma, all semi-ample models are crepant birationally equivalent to each other.

	\subsection{The canonical bundle formula}
	In this subsection, we give a version of the canonical bundle formula that follows from the work of Kawamata, Fujino--Mori, Ambro, and Koll\'{a}r (cf. \cite{Kaw98}, \cite{FM00}, \cite{Amb05} and \cite{Kol07}).
	\begin{defn}\label{lc-trivial fibration}
		Let $X, Z$ be normal varieties and $f: X\rightarrow Z$ a fibration. Let $\Delta$ be a $\mathbb{Q}$-divisor on $X$ such that $K_X+\Delta$ is $\mathbb{Q}$-Cartier. We call $(X,\Delta)\rightarrow Z$ an lc-trivial fibration if
		\begin{itemize}
			\item $(X_\eta,\Delta_{\eta})$ is sub-lc over the generic point of $Z$,
			\item $K_X+\Delta\sim_{\mathbb{Q},Z} 0$, and
			\item $h^0(X_g,\mathcal{O}_{X_g}( \lceil \Delta_{X_g,< 0}\rceil  ) )=1$, where $X_g$ is a general fiber of $f$.			
		\end{itemize}
	\end{defn}

	\begin{thm}[The canonical bundle formula]
		\label{canonical bundle formula}
		Let $X, Z$ be normal projective varieties and $\Delta$ a $\mathbb{Q}$-divisor on $X$ such that $K_X+\Delta$ is $\mathbb{Q}$-Cartier and $(X,\Delta)\rightarrow Z$ is an lc-trivial fibration.
		Suppose $B$ is a reduced divisor on $Z$ such that 
		if $D$ is prime divisor not contained in $B$, then
		\begin{itemize}
			\item no component of $\Delta$ dominates $D$, and
			\item $(X,\Delta+f^*D)$ is lc over the generic point of $D$.
		\end{itemize}
		Then one can write 
		$$K_X+\Delta\sim_{\mathbb{Q}} f^*(K_Z+B_Z+\M_Z),$$
		where $(Z, B_Z+\M_Z)$ is a generalized log pair satisfying the following:
		\begin{enumerate}
			\item[(a)] $B_Z$ is a $\mathbb{Q}$-divisor supported on $B$, called the boundary part.
			\item[(b)] Let $P\subset B$ be an irreducible divisor. Then
			$$\mathrm{coeff}_P(B_Z)=\mathrm{sup}_Q\{1-\frac{1+a(Q,X,\Delta)}{\mult_Qf^*P}\}$$
			where the supremum is taken over all divisors over $X$ that dominate $P$. Note that it is exactly the log canonical threshold of $f^*P$ with respect to $(X,\Delta)$ over the generic point of $P$.
			\item[(c)] $\M_Z=M(X/Z,\Delta)$ is the moduli part. It depends only on the crepant birationally equivalent class of $(F,\Delta|_F)$ and $Z$, where $F$ is the generic fiber of $f$. It is a $\mathbf{b}$-divisor and it is the push-forward of a nef $\mathbb{Q}$-divisor $\M_{Z'}=M(X'/{Z'},\Delta')$ of some projective birational morphism $Z'\rightarrow Z$, 
			where $X'\rightarrow X\times_Z Z'$ is birational onto the main component with $X'$ normal and projective. We say that the $\mathbf{b}$-divisor $\M$ descends on $Z'$.
			\item[(d)] 
			If $X\rightarrow Z$ and $\Delta,B$ satisfy the standard normal crossing assumptions, see Definition \ref{standard normal crossing}, then $\M$ descends on $Z$.
		\end{enumerate}
		Furthermore, by definition of $B_Z$, we have the following:
		\begin{enumerate}
			\item[(e)] If $P$ is dominated by a divisor $E$ such that $a(E,X,\Delta)<0$ (resp. $\leq 0$) then $\mathrm{coeff}_P(B_Z)>0$ (resp. $\geq 0$).
			
			\item[(f)] If $ \Delta$ is effective then so is $B_Z$.
			
			\item[(g)] $\mathrm{coeff}_P(B_Z)=1$ if and only if $P$ is dominated by a divisor $E$ such that $a(E,X,\Delta)=-1$.
		\end{enumerate}
	\end{thm}
	\begin{defn}[Standard normal crossing assumptions]\label{standard normal crossing}
		We say that $f: X\rightarrow Z$ and the $\mathbb{Q}$-divisors $\Delta, B$ satisfy the standard normal crossing assumptions if the following hold,
		\begin{itemize}
			\item $(X,\Delta+f^*B)$ and $(Z,B)$ are log smooth, and
			\item $(X,\Delta)$ is log smooth over $Z\setminus B$.
		\end{itemize}
	\end{defn}
	\begin{lem}\label{inverse of canonical bundle formula}
		Let $f:(X,\Delta)\rightarrow Z$ be an lc-trivial fibration. Suppose by the canonical bundle formula, 
		$$K_X+\Delta\sim_{\mathbb{Q}}f^*(K_Z+B_Z+\M_Z).$$ 
		Then
		\begin{enumerate}
			\item[(a)] every generalized log canonical center of $(Z,B_Z+\M_Z)$ is dominated by a log canonical center of $(X,\Delta)$, 
			\item[(b)] every $f$-vertical log center of $(X,\Delta)$ dominates a generalized log center of $(Z,B_Z+\M_Z)$,
			\item[(c)] $(X,\Delta)$ is sub-lc if and only if $(Z,B_Z+\M_Z)$ is generalized sub-lc, and
			\item[(d)] let $a\in (0,1)$ be a rational number. If $(Z,B_Z+\M_Z)$ is generalized sub-$a$-lc, then $a(E,X,\Delta)\geq -1+a$ for any prime divisor $E$ over $X$ whose center is vertical over $Z$. 
		\end{enumerate}
		\begin{proof}
			For any prime divisor $Q$ over $X$ whose image on $X$ does not dominate $Z$, by \cite[Lemma 2.45]{KM98}, there is a projective birational morphism $W\rightarrow Z$ such that the strict transform of $Q$ on the normalization of the main component of $X\times_Z W$ dominates a prime divisor on $W$. We call this prime divisor the ``divisorial image" of $Q$.
			
			Let $E$ be any prime divisor over $Z$. By \cite[Theorem 2.1, Proposition 4.4, and Remark 4.5]{AK00}, we have the following diagram
			$$\xymatrix{
				(X,\Delta)  \ar[d] _{f}&    (X',\Delta') \ar[l]_{\rho_X} \ar[d]^{f'}\\
				Z&    Z'\ar[l]_{\rho}	 
			} $$
			where
			\begin{itemize}
				\item $E$ is a prime divisor on $Z'$,
				\item $\rho$ and $\rho_X$ are birational, 
				\item $K_{X'}+\Delta'\sim_{\mathbb{Q}}\rho_X^*(K_X+\Delta)$,
				\item $Z'$ is smooth, $X'$ is normal projective with quotient singularities, and
				\item $f'$ is equidimensional.
			\end{itemize}
			
			(a). Suppose $E$ is a prime divisor over $Z$ such that $a(E,Z,B_Z+\M_Z)=-1$. 
			Write $K_{Z'}+B_{Z'}+\M_{Z'}\sim_{\mathbb{Q}}\rho^*(K_Z+B_Z+\M_Z)$, because $a(E,Z,B_Z+\M_Z)=-1$, then $\mathrm{coeff}_{E}(B_{Z'})=1$. By the definition of the boundary part, $$\mathrm{coeff}_E(B_{Z'})=\mathrm{sup}_P\{1-\frac{1+a(P,X',\Delta')}{\mult_Pf'^*E}\}$$
			then there is a prime divisor $Q$ on $X'$ dominating $E$ such that $a(Q,X',\Delta')=-1$. Also because $Q$ dominates $E$, $\rho_X(Q)$ is an lc center of $(X,\Delta)$ which dominates $\rho(E)$.

			(b). Suppose $Q$ is a prime divisor over $X$ such that $a(Q,X,\Delta)\in [-1,0)$ and $Q$ does not dominate $Z$. Let $E$ be the divisorial image of $Q$. By the diagram, it is easy to see that $Q$ is a prime divisor on $X'$.
			By assumption, $a(Q,X',\Delta')<0$, suppose $E$ is the prime divisor on $Z'$ dominated by $Q$, then 
			$$\mathrm{coeff}_E(B_{Z'})=\mathrm{sup}_P\{1-\frac{1+a(P,X',\Delta')}{\mult_Pf'^*E}\}\geq 1-\frac{1+a(Q,X',\Delta')}{\mult_Qf'^*E}>0,$$
			which means $E$ is a generalized log center of $(Z,B_Z+\M_Z)$.
			
			(c). Consider the diagram constructed above. By the definition of the boundary part
			$$\mathrm{coeff}_E(B_{Z'})=\mathrm{sup}_P\{1-\frac{1+a(P,X',\Delta')}{\mult_Pf'^*E}\},$$
			then $a(E,Z,B_Z+\M_Z)\geq -1$ if and only if $\mathrm{coeff}_E(B_{Z'})\leq -1$ if and only if $a(P,X',\Delta')=a(P,X,\Delta)\geq -1$ for every prime divisor $P$ over $X$ that dominates $E$. Because $E$ is arbitrary, then $(X,\Delta)$ is sub-lc implies $(Z,B_Z+\M_Z)$ is generalized sub-lc.
			
			Conversely, let $Q$ be any prime divisor over $X$. If the image of $Q$ on $X$ dominates $Z$, then $a(Q,X',\Delta')\geq -1$ because $(X_\eta,\Delta_\eta)$ is sub-lc according to the definition of lc-trivial fibration. If the image of $Q$ on $X$ does not dominate $Z$, let $E$ be the divisorial image of $Q$. Then for the same reason as in the proof of (b), $(Z, B_Z+\M_Z)$ is generalized sub-lc implies that $(X,\Delta)$ is sub-lc.
			
			(d). Let $Q$ be any prime divisor over $X$ whose image on $X$ does not dominates $Z$, $E$ be the divisorial image of $Q$. Because $(Z,B_Z+\M_Z)$ is generalized sub-$a$-lc, then
			$$1-a\geq \mathrm{coeff}_E(B_{Z'})=\mathrm{sup}_P\{1-\frac{1+a(P,X',\Delta')}{\mult_Pf'^*E}\}\geq 1-\frac{1+a(Q,X,\Delta)}{\mult_Qf'^*E}.$$
			Because $Z'$ is smooth, $E$ is a Cartier divisor, then $\mult_Qf'^*E \geq 1$, $a(Q,X,\Delta) \geq -1+a$.
		\end{proof}
	\end{lem}

	\section{Polarized log Calabi--Yau pairs}
	In this section, we recall some definitions and results in \cite{Kol23}, \cite{Bir22} and \cite{Bir23}. First we begin with definitions and results regarding stable families from Koll\'ar \cite{Kol23}.
	
	Let $f:X\rightarrow S$ be a flat morphism of schemes with $S_2$ fibers of pure dimension. A closed subscheme $D\subset X$ is a relative Mumford divisor over $S$ if there is an open subset $U\subset X$ such that
	\begin{itemize}
		\item codimension of $X_s\setminus U_s$ is $\geq 2$ for every $s\in S$,
		\item $D|_U$ is a Cartier divisor,
		\item $\Supp(D|_U)$ does not contain any irreducible component of any fiber $U_s$,
		\item $D$ is the closure of $D|_U$, and
		\item $X\rightarrow S$ is smooth at the generic points of $X_s\cap D$ for every $s\in S$.
	\end{itemize}

	A locally stable family $f:(X,\Delta)\rightarrow S$ is given by a morphism $f:X\rightarrow S$ of schemes and closed subschemes $D_i\subset X$ for $i=1,\cdots , m$ where 
	\begin{itemize}
		\item $f$ is flat and projective with reduced geometrically connected $S_2$ fibers of pure dimension whose codimension one singularities are nodal,
		\item $D_i$ are relative Mumford divisors over $S$,
		\item $\Delta=\sum a_iD_i$, 
		\item $K_X+\Delta$ is $\mathbb{Q}$-Cartier, and
		\item for each $s\in S$, $(X_s,\Delta_s)$ is an slc pair over the residue field $k(s)$ where $\Delta_s=\sum a_iD_{i,s}$.
	\end{itemize}
	A stable family $f:(X,\Delta)\rightarrow S$ is a locally stable family such that 
	\begin{itemize}
		\item for each $s\in S$, $K_{X_s}+\Delta_s$ is ample.
	\end{itemize}
	
	Next we recall some definition regarding the moduli space of polarized log Calabi--Yau pairs from \cite{Bir23}.
	
	A polarized log Calabi--Yau pair consists of a connected projective slc Calabi--Yau pair $(X,\Delta)$ and an ample divisor $N\geq 0$ such that $(X,\Delta+uN)$ is slc for any sufficiently small positive rational number $u>0$. Notice that when $(X,\Delta)$ is klt such $u$ exists naturally. We refer to such a slc pair by saying $(X,\Delta), N$ is a polarized log Calabi--Yau pair. 
	Fix a natural number $d$ and positive rational numbers $c,v,$. If in addition $\mathrm{dim}X=d,\Delta=cD$ for some divisor $D$ and $\mathrm{vol}(N)=v$, we call $(X,\Delta),N$ a $(d,c,v)$-polarized log Calabi--Yau pair. For the definition of slc pairs, see \cite[Definition-Lemma 5.10]{Kol13}.
	
	The following is Definition 1.9 in the first arxiv version of \cite{Bir23}.
	\begin{defn}
		Let $S$ be a reduced scheme. A $(d,c,v)$-polarized log Calabi--Yau family over $S$ consists of a projective morphism $f:X\rightarrow S$ of schemes, a $\mathbb{Q}$-divisor $\Delta$ and a divisor $N$ on $X$ such that 
		\begin{itemize}
			\item $(X,\Delta+uN)\rightarrow S$ is a stable family for some rational number $u>0$ with fibers of pure dimension $d$,
			\item $\Delta=cD$ where $D\geq 0$ is a relative Mumford divisor,
			\item $N\geq 0$ is a relative Mumford divisor,
			\item $K_{X/S}+\Delta\sim_{\mathbb{Q}}0/S$, and
			\item for any fiber $X_s$ of $f$, $\mathrm{vol}(N|_{X_s})=v$.
		\end{itemize}
	\end{defn}

	The following is Lemma 7.7 in the first arxiv version of \cite{Bir23}. 
	\begin{lem}\label{Bir23 Lemma 7.7}
		Let $d$ be a natural number and $c,v$ be positive rational numbers. Then there exists a positive rational number $t$ and a natural number $r$ such that $rc,rt\in \mathbb{N}$ satisfying the following. Assume $(X,\Delta),N$ is a $(d,c,v)$-polarized log Calabi--Yau pair. Then
		\begin{itemize}
			\item $(X,\Delta+tN)$ is slc,
			\item $\Delta+tN$ uniquely determines $\Delta,N$, and
			\item $r(K_X+\Delta+tN)$ is very ample with $$h^j(mr(K_X+\Delta+tN))=0$$
			for $m,j>0$.
		\end{itemize}
	\end{lem}
	
	The following is 7.6 in the first arxiv version of \cite{Bir23}. 
	\begin{defn}
		Let $d,c,v,t,r$ be as in Lemma \ref{Bir23 Lemma 7.7}, $n$ a natural number. To simplify the notation, define $\Xi=(d,c,v,t,r,\mathbb{P}^n)$. Let $S$ be a reduced scheme. A strongly embedded $\Xi$-polarized log Calabi--Yau family over $S$ is a $(d,c,v)$-polarized log Calabi--Yau family $f:(X,\Delta),N\rightarrow S$ together with a closed embedding $g:X\rightarrow \mathbb{P}^n_S$ such that
		\begin{itemize}
			\item $(X,\Delta+tN)\rightarrow S$ is a stable family,
			\item $f=\pi \circ g$ where $\pi$ denotes the projection $\mathbb{P}^n _S\rightarrow S$,
			\item letting $\mathcal{L}:=g^*\mathcal{O}_{\mathbb{P}^n_S}(1)$, we have $R^qf_*\mathcal{L}\cong R^q\pi_* \mathcal{O}_{\mathbb{P}^n_S}(1)$ for all $q$, and
			\item for every $s\in S$, we have
			$$\mathcal{L}_s\cong \mathcal{O}_{X_s}(r(K_{X_s}+\Delta_s+tN_s)).$$
		\end{itemize}
		Denote the functor $\mathcal{E}^s\mathcal{PCY}_{\Xi}$ on the category of reduced schemes by setting
		$$\mathcal{E}^s\mathcal{PCY}_{\Xi}(S)=\{\text{strongly embedded }\Xi\text{-polarized log Calabi--Yau families over }S\}.$$
	\end{defn}
	
	The following is Proposition 7.8 in the first arxiv version of \cite{Bir23}. Also see \cite{Bir22} and \cite{Kol23}.
	\begin{thm}[{\cite[Proposition 7.8]{Bir23}}]
		The functor $\mathcal{E}^s\mathcal{PCY}_{\Xi}$ has a fine moduli space $\Cs$, which is a reduced separated scheme of finite type, and a universal family $(\Cx\subset \mathbb{P}^n_{\Cs},\Cd),\Cn\rightarrow \Cs$
	\end{thm}

	\begin{rem}\label{Parametrizing space for log pairs}
		Fix a positive integer $d$ and two positive rational numbers $c,v$. Let $f:X\rightarrow Z$ be a fibration, $\Delta$ a $\mathbb{Q}$-divisor and $N$ a divisor on $X$. Suppose
		\begin{itemize}
			\item a general fiber $(X_g,\Delta_g)$ is a $d$-dimensional klt pair,
			\item $\mathrm{coeff}\Delta_g\subset c\mathbb{N}$,
			\item $K_{X_g}+\Delta_g\sim_{\mathbb{Q},Z}0$, and
			\item $N_g$ is ample and $\vol(N_g)=v$.
		\end{itemize}
		It is easy to see that there is a dense open subset $U\subset Z$ such that $(X_U,\Delta_U),N_U\rightarrow U$ is a $(d,c,v)$-polarized log Calabi--Yau family over $U$. By Lemma \ref{Bir23 Lemma 7.7}, there exist a positive rational number $t$ and a natural number $r$ such that $(X_u,\Delta_u+tN_u)$ is slc and $r(K_{X_u}+\Delta_u+tN_u)$ is very ample without higher cohomology for every closed point $u\in U$. Then by cohomology and base change, $r(K_{X_U}+\Delta_U+tN_U)$ is relatively very ample, and after replacing $U$ by a dense open subset, it defines a closed embedding $g: X_U\hookrightarrow \mathbb{P}^n_U$, where $n$ is a natural number depending only on $d,c,v,t,r$. 
		
		It is easy to see that $(X_U,\Delta_U+tN_U)\rightarrow U$ is a stable family, then $f_U:(X_U\subset \mathbb{P}^n_U,\Delta_U),N_U\rightarrow U$ together with $g: X_U\hookrightarrow \mathbb{P}^n_U$ is a strongly embedded $(d,c,v,t,r,\mathbb{P}^n)$-polarized log Calabi--Yau family over $U$. Since $\Ce^s\mathcal{PCY}_{\Theta}$ has a fine moduli space with the universal family $(\Cx\subset \mathbb{P}^n_{\Cs},\Cd),\Cn\rightarrow \Cs$, we have $(X_U,\Delta_U)\cong (\Cx,\Cd)\times_\Cs U$, where $U\rightarrow \Cs$ is the moduli map defined by $f|_{X_U}$.
	\end{rem}

	\section{Lc-trivial fibrations}
	
	\begin{lem}
		\label{general fiber calabi-yau will induce lc-trivial fibration structure}
		
		Let $(X,\Delta)$ be a projective lc pair, $f:X\rightarrow Z$ a fibration to a projective normal $\mathbb{Q}$-factorial variety $Z$ such that $K_{X_g}+\Delta_g\sim _{\mathbb{Q}}0$, where $X_g$ is a general fiber of $f$. Let $g:(X',\Delta')\rightarrow (X,\Delta)$ be a crepant birational morphism and $D$ be a reduced divisor on $Z$ such that the morphism $h:=f\circ g:(X',\Delta')\rightarrow Z$ is log smooth over $Z\setminus D$.
		
		Then there is a $\mathbb{Q}$-divisor $\overline{\Delta'}$ on $X'$, such that 
		\begin{itemize}
			\item $\overline{\Delta'}_g=\Delta'_g$,
			\item $(X',\overline{\Delta}')$ is log smooth over $Z\setminus D$, and
			\item $(X',\overline{\Delta}')\rightarrow Z$ is an lc-trivial fibration.
		\end{itemize}
		\begin{proof}
			Let $\eta $ denote the generic point of $Z$. Since $(X_g,\Delta_g)$ is a log Calabi--Yau variety, then $K_{X'_\eta}+\Delta_\eta' \sim_{\mathbb{Q}} 0$, which means that there exists a vertical $\mathbb{Q}$-divisor $B'$ such that $K_{X'}+\Delta'+B'\sim_{\mathbb{Q},Z} 0$.

			Write $B'=R+G$, such that each irreducible component of $\Supp(R)$ is not contained in $h^{-1}(\mathrm{Supp}(D))$ and each irreducible component of $\Supp(G)$ is contained in $h^{-1}(\mathrm{Supp}(D))$. Suppose $R=\sum a_iR_i$, where $R_i$ are distinct prime divisors. Because each $R_i$ is vertical, $Z$ is $\mathbb{Q}$-factorial and $h$ is smooth over $Z\setminus \mathrm{Supp}(D)$, hence smooth at the generic point of $R_i$, then each $R_i$ dominates a prime $\mathbb{Q}$-Cartier divisor $R_{Z,i}$ on $Z$. Let $R_Z:=\sum a_i R_{Z,i}$, then there exists a $\mathbb{Q}$-divisor $F_R$ supported on $h^{-1}(\mathrm{Supp}(D))$, such that $R+F_R=h^*R_Z$ and we have $K_{X'}+\Delta'+B'-(R+F_R)\sim_{\mathbb{Q},Z}K_{X'}+\Delta'+G-F_R\sim_{\mathbb{Q},Z}0$. 
			
			Let $\overline{\Delta}':=\Delta'+G-F_R$, then $K_{X'}+ \overline{\Delta}'\sim_{\mathbb{Q},Z} 0$ and $\overline{\Delta}'_g= \Delta'_g$. Write $\Delta'_g=\Delta'_{g,> 0}-\Delta'_{g,< 0}$. Since $\Delta'_{g,< 0}$ is $g$-exceptional, it is easy to see that $(X',\overline{\Delta}')\rightarrow Z$ is an lc-trivial fibration. Because $(X',\Delta')$ is log smooth over $Z\setminus D$, $\mathrm{Supp}(F_R)\subset h^{-1}(D)$ and $\mathrm{Supp}(G)\subset h^{-1}(D)$, then $(X',\overline{\Delta}')$ is log smooth over $Z\setminus D$. 
		\end{proof}
	\end{lem}
	Lemma \ref{general fiber calabi-yau will induce lc-trivial fibration structure} implies that for an lc pair $(X,\Delta)$ and a morphism $f:X\rightarrow Z$, if a general fiber $(X_g,\Delta_g)$ is a log Calabi--Yau pair, then $(X,\Delta)\rightarrow Z$ induces an lc-trivial fibration, hence it defines a moduli $\mathbf{b}$-divisor $\M$. Furthermore, if there exists a reduced divisor $D\subset Z$ such that $(X,\Delta)\rightarrow Z$ has a fiberwise log resolution over $Z\setminus D$, then it defines an lc-trivial fibration $(X',\Delta')\rightarrow Z$ such that $(X',\Delta')$ is log smooth over $Z\setminus D$. Therefore if $(Z,D)$ is log smooth, then $X'\rightarrow Z$ and $\Delta',D$ satisfy the standard normal crossing assumptions; hence by Theorem \ref{canonical bundle formula}, the moduli $\mathbf{b}$-divisor $\M$ descends on $Z$.
	
	\begin{prop}[{\cite[Proposition 3.1]{Amb05}}]
		\label{moduli part stable under base change}
		Suppose $f:(X,\Delta)\rightarrow Z$ is an lc-trivial fibration. Let $\rho: Z'\rightarrow Z$ be a surjective morphism from a proper normal variety $Z'$ and $f':(X',\Delta')\rightarrow Z'$ the lc-trivial fibration induced by the normalization of the main component of the base change.
		$$\xymatrix{
			(X,\Delta)  \ar[d] _{f}&    (X',\Delta') \ar[l]_{\rho_X} \ar[d]^{f'}\\
			Z&    Z'\ar[l]_{\rho}	. 
		}$$
		Let $\M$ and $\M'$ be the corresponding moduli $\mathbf{b}$-divisors. If $\M$ descends on $Z$, then $\M'$ descends on $Z'$ and
		$$\rho^*\M_Z=\M'_{Z'}.$$
	\end{prop}
	
	\begin{thm}[{\cite[Theorem 6.5]{Jia21}}]\label{geneic point maps to log smooth locus means moduli divisor of restriction is restriction of moduli divisor}
		Let $\Ct$ be a projective variety, $(\Cy_\Ct,\Cr_\Ct)\rightarrow \Ct$ an lc-trivial fibration such that the corresponding moduli divisor $\Mm_{\Ct}$ descends on $\Ct$. Suppose there exists an open subset $\Ct^o\hookrightarrow \Ct$ such that $(\Cy_\Ct,\Cr_\Ct)$ is log smooth over $\Ct^o$. Let $\phi:Z\rightarrow \Ct$ be a morphism such that $\phi$ maps the generic point of $Z$ into $\Ct^o$. If $(X,\Delta)\rightarrow Z$ is an lc-trivial fibration whose generic fiber is crepant birationally equivalent to generic fiber of $(\Cy_\Ct,\Cr_\Ct)\times_{\Ct}Z\rightarrow Z$. 
		Let $\M_Z$ be the moduli $\mathbf{b}$-divisor of $f$. If $\M_Z$ descends on $Z$, then we have
		$$\M_Z=\phi^*\Mm_{\Ct}.$$
	\end{thm}
	
	\begin{thm}[{\cite[Theorem 3.3]{Amb05}} and {\cite[Theorem 6.5]{Jia21}}]
		\label{moduli part is nef and good}
		Let $f:(X,\Delta)\rightarrow S$ be an lc-trivial fibration such that $\Delta_g$ is effective, then there exists a commutative diagram
		$$\xymatrix{
			(X,\Delta)  \ar[d] _{f}& &   (X_T,\Delta_T) \ar[d]_{f_T} &\\
			S 	\ar@/_1pc/@{-->}[rrr] _\Phi &  \bar{S}\ar[l]_{\tau} \ar[r]^{\rho}& 	 T \ar@/_1pc/@{-->}[ll]_i  \ar[r]^{\pi} & V
		} $$
		such that
		\begin{itemize}
			\item $f_T:(X_T,\Delta_T)\rightarrow T$ is an lc-trivial fibration,
			\item $\tau$ and $\pi$ are generically finite and surjective morphisms, $\rho$ is surjective,
			\item there exists a nonempty open subset $U\subset \bar{S}$ and an isomorphism
			$$\xymatrix{
				(X,\Delta)\times_S U  \ar[rr] _{\cong}\ar[dr]& &   (X_T,\Delta_T)\times_{T} U \ar[dl]\\
				&  U,& 	 
			} $$
			\item let $\M$, $\N$ be the corresponding moduli-$\bm{b}$-divisors of $f$ and $f_T$, then $\N$ is $\bm{b}$-nef and big,
			\item if $\M$ descends on $S$ and $\N$ descends on $T$, then $\tau^*\M_S=\rho^*\N_{T}$,
			\item $\Phi:S\dashrightarrow V$ is the period map defined in \cite[Section 2]{Amb05}, and 
			\item $i:T\dashrightarrow S$ is a rational map such that the generic fiber of $f_T$ is equal to the pullback of $f$ via $i$.
		\end{itemize}
		
	\end{thm}

	\section{Weak boundedness}
	The definition of weak boundedness is first introduced in \cite{KL10}. We find this definition useful in proving the birational boundedness of fibrations with bounded general fibers.
	\begin{defn}
		A $(g,m)$-curve is a smooth curve $C^o$ whose smooth compactification $C$ has genus $g$ and such that $C\setminus C^o$ consists of $m$ closed points.
	\end{defn}
	\begin{defn}\label{definition of weakly bounded}
		Let $W$ be a proper scheme with a line bundle $\mathcal{N}$ and let $U$ be an open subset of a proper variety. We say a morphism $\xi:U\rightarrow W$ is weakly bounded with respect to $\mathcal{N}$ if there exists a function $b_{\mathcal{N}}:\mathbb{Z}^2_{\geq 0}\rightarrow \mathbb{Z}$ such that for every pair $(g,m)$ of non-negative integers, for every $(g,m)$-curve $C^o\subseteq C$, and for every morphism $C^o\rightarrow U$, one has that $\mathrm{deg}\xi_C ^*\mathcal{N}\leq b_{\mathcal{N}}(g,m)$, where $\xi_C:C\rightarrow W$ is the induced morphism. The function $b_{\mathcal{N}}$ is called a weak bound, and we say that $\xi$ is weakly bounded by $b_\mathcal{N}$.
		
		We say an open subset $U$ of a proper variety is weakly bounded if there exists a compactification $i: U\hookrightarrow W$, such that $i: U\hookrightarrow W$ is weakly bounded with respect to an ample line bundle $\mathcal{N}$ on $W$. 
	\end{defn}
	\begin{lem}[{\cite[Lemma 5.3]{Jia21}}] \label{decompose to get weakly bounded} 
		Let $T$ be a quasi-projective variety. Then we can decompose $T$ into finitely many locally closed subsets $\cup_{i\in I} T_i$, such that each $T_i$ is weakly bounded.
		
	\end{lem} 
	The following theorem says that morphisms from bounded varieties to weakly bounded varieties can be parametrized by a scheme of finite type. 
	\begin{thm}[{\cite[Proposition 2.14]{KL10}}]\label{weakly bounded morphisms are bounded}
		Let $T$ be a quasi-compact separated reduced $\mathbb{C}$-scheme and $\mathscr{U}\rightarrow T$ a smooth morphism. Given a projective $T$-variety and a polarization over $T$, $(\mathscr{M},\mathcal{O}_{\mathscr{M}}(1))$, an open subvariety $\mathscr{M}^o\hookrightarrow \mathscr{M}$ over $T$, and a weak bound $b$, there exists a $T$-scheme of finite type $\mathscr{W}^b_{\mathscr{M}^o}$ and a morphism $\Theta:\mathscr{W}^b_{\mathscr{M}^o}\times \mathscr{U}\rightarrow \mathscr{M}^o$ such that for every closed point $t\in T$ and for every morphism $\xi:\mathscr{U}_t\rightarrow \mathscr{M}^o_t\subset \mathscr{M}_t$ that is weakly bounded by $b$ there exists a point $p\in \mathscr{W}^b_{\mathscr{M}^o_t}$ such that $\xi=\Theta|_{\{p\}\times \mathscr{U}_t}$.
	\end{thm}
	
	In particular, suppose $U$ is an open subset of a proper variety and it is weakly bounded with respect to the compactification $U\hookrightarrow W$, the ample line bundle $\mathcal{N}$ on $W$ and the function $b$. We define $\mathscr{M}:=W\times T$ and $\mathscr{M}^o:=U\times T$. By definition, every morphism $\mathscr{U}_t\rightarrow U$ is equal to the morphism $\xi:\mathscr{U}_t\rightarrow \mathscr{M}^o_t\subset \mathscr{M}_t$ which is weakly bounded by $b$, then by Theorem \ref{weakly bounded morphisms are bounded}, $\xi=\Theta|_{\{p\}\times \mathscr{U}_t}$ for a closed point $p\in \mathscr{W}^b_{\mathscr{M}^o_t}$.

	\section{Birational boundedness of fibrations}
	
	The purpose of this section is to prove Theorem \ref{bounded base implies bounded fibration}.
	
	\begin{defn}\label{definition of bounded base}
		Let $d$ be a positive integer, $\Ci\subset [0,1]\cap\mathbb{Q}$ a DCC set and $v,V$ two positive rational numbers. Let $\Cp(d,\Ci,v,V)$ be the set of log pairs $(X,\Delta)$ satisfying the following
		\begin{itemize}
			\item $(X,\Delta)$ is a $d$-dimensional projective klt pair,
			\item $\mathrm{coeff}(\Delta)\subset \Ci$,
			\item there is a fibration $f:X\rightarrow Z$ such that $K_X+\Delta\sim_{\mathbb{Q},Z}0$, by the canonical bundle formula, we write $K_X+\Delta\sim_{\mathbb{Q}}f^*(K_Z+B_Z+\M_Z)$,
			\item $K_Z+B_Z+\M_Z$ is nef,
			\item there is a divisor $A$ on $X$ such that $A_g:=A|_{X_g}$ is ample and $\vol(A_g)=v$, where $X_g$ is a general fiber of $f$, and
			\item there is a divisor $H$ on $Z$ such that $\vol(H+K_Z+B_Z+\M_Z)\leq V$ and $|H|$ defines a birational map.
		\end{itemize}
	\end{defn}
	
	\begin{thm}\label{bounded base implies bounded fibration}
		Let $d$ be a positive integer, $\Ci\subset [0,1]\cap\mathbb{Q}$ a DCC set and $v,V$ two positive rational numbers. Then 
		$$\Cp(d,\Ci,v,V)$$
		is birationally bounded. 
		
		Furthermore, there is a log smooth morphism $(\Cy,\Cq)\rightarrow S$ with reduced boundary depending only on $d,\Ci,v,V$, where $S$ is of finite type, such that for any $(X,\Delta)\in \Cp(d,\Ci,v,V)$, there is a closed point $s\in S$ and a birational map $\Cy_s\dashrightarrow X$ such that $\Cq_s\geq \mathbf{L}_{\Delta,\Cy_s}$.
	\end{thm}
	
	Recall that $\mathbf{L}_{\Delta}$ is the $\mathbf{b}$-divisor defined in Definition \ref{b-divisors}.
	
	The proof of Theorem \ref{bounded base implies bounded fibration} is pretty long, we divide it into several middle results. First we construct a family with nice properties which parametrizes a general fiber of $f:(X,\Delta)\rightarrow Z$.
	
	\begin{thm}\label{essential diagram}
		Let $d$ be a natural number and $c,v$ positive rational numbers. Let $t,r$ be the numbers defined in Lemma \ref{Bir23 Lemma 7.7}. Then there is a commutative diagram 
		$$\xymatrix{
			(\Cx_{\Cs^o},\Cd_{\Cs^o})  \ar[d] _{\Cf}& (\Cx_{\bar{\Cs^o}},\Cd_{\bar{\Cs^o}})\ar[l]_{\tau_\Cx} \ar[r]^{\rho_{\Cx}} \ar[d] &   (\Cx_{{\Ct^o}},\Cd_{{\Ct^o}}) \ar[d]_{\Cf_{{\Ct^o}}} &\\
			\Cs^o 	\ar@/_1pc/[rrr] _\Phi &  \bar{\Cs^o}\ar[l]_{\tau} \ar[r]^{\rho}& 	 {\Ct^o}   \ar[r]^{\pi} &\Cv^o,
		}$$
		where
		\begin{itemize}
			\item $(\Cx_{\Cs^o},\Cd_{\Cs^o})\rightarrow {\Cs^o},(\Cx_{\Ct^o} ,\Cd_{\Ct^o})\rightarrow {\Ct^o}$ and $(\Cx_{\bar{\Cs^o}},\Cd_{\bar{\Cs^o}})\rightarrow \bar{\Cs}^o$ are locally stable families of log Calabi--Yau pairs over quasi-projective varieties,
			\item $(\Cx_{\Cs^o},\Cd_{\Cs^o})\rightarrow {\Cs^o}$ and $(\Cx_{\Ct^o} ,\Cd_{\Ct^o})\rightarrow {\Ct^o}$ have fiberwise log resolutions,
			\item $\Phi,\rho$ are surjective, $\tau$ is finite, $\pi$ is \'etale,
			\item $(\Cx_{\bar{{\Cs^o}}},\Cd_{\bar{{\Cs^o}}})\cong (\Cx_{\Cs^o},\Cd_{\Cs^o})\times_{\Cs^o} \bar{{\Cs^o}}\cong (\Cx_{\Ct^o},\Cd_{\Ct^o})\times_{\Ct^o} \bar{{\Cs^o}}$,
			\item there is a family of effective relative ample over ${\Cs^o}$ divisors $\{\Gamma_{{\Cs^o},\alpha},\alpha\in \mathscr{C}\}$ on $\Cx_{\Cs^o}$ such that $\tau_{\Cx}^* \Gamma_{{\Cs^o},\alpha} =\rho_\Cx^* \Gamma_{{\Ct^o},\alpha}$ for effective relative ample over ${\Ct^o}$ divisors $\{\Gamma_{{\Ct^o},\alpha},\alpha\in \mathscr{C}\}$ on $\Cx_{\Ct^o}$, where $\mathscr{C}$ is a variety with positive dimension,
			\item $\bigcap_{\alpha\in \mathscr{C}} \Gamma_{{\Cs^o},\alpha}=\bigcap_{\alpha\in \mathscr{C}} \Gamma_{{\Ct^o},\alpha}=\emptyset$, 
			\item $\Cv^o$ is weakly bounded with respect to an effective ample divisor $\Ca$ on a projective compactification $\Cv^o\hookrightarrow\Cv$ of $\Cv^o$,
			\item $\Ct^o$ has a projective compactification $\Ct^o\hookrightarrow \Ct$,
			\item $\pi$ extends to a generically finite morphism $\pi:\Ct\rightarrow \Cv$, 
			\item the moduli $\mathbf{b}$-divisor $\Nn$ of $\mathcal{F}_{\Ct^o}$ descends on $\Ct$,
			\item $l\Nn_{\Ct}$ is Cartier for a natural number $l$ depending only on $d,c,v$, and
			\item we can choose a section of $l\Nn_{\Ct}$ so that $l\Nn_{\Ct}\geq \pi^* \Ca$.
		\end{itemize}
		And if $f:X_U\rightarrow U$ is a fibration, $\Delta_U$ is a $\mathbb{Q}$-divisor and $N_U$ is a divisor on $X_U$, such that the general fiber $X_g$ is klt and $(X_g,\Delta_g),N_g$ forms a $(d,c,v)$-polarized log Calabi--Yau pair, then after replacing $U$ by an open subset, there is a morphism $\phi_U:U\rightarrow \Cs^o$ such that 
		$$(X_U,\Delta_U)\cong (\Cx_{\Cs^o},\Cd_{\Cs^o})\times _{\Cs^o} U.$$
		Suppose $U\hookrightarrow Z$ is a compactification of $U$ such that the moduli $\mathbf{b}$-divisor $\M$ descends on $Z$, $\Phi\circ \phi_U$ extends to a morphism $\phi_{\Cv}:Z\rightarrow \Cv$, then $l\M_Z$ is Cartier and we can choose a section of $l\M_Z$ such that $l\M_Z\geq \phi_{\Cv}^*\Ca$.
		
		\begin{proof}
			Let $n$ be the natural number defined in Lemma \ref{Parametrizing space for log pairs} and $(\Cx\subset \mathbb{P}^n,\Cd),\Cn\rightarrow \Cs$ the universal family of the functor $\mathcal{E}^s\mathcal{PCY}_{\Xi}$, where $\Xi:=(d,c,v,t,r,\mathbb{P}^n)$. Because $(\Cx,\Cd)\rightarrow \Cs$ is a family of log Calabi--Yau pairs, by Theorem \ref{moduli part is nef and good}, we have the following diagram
			$$\xymatrix{
				(\Cx,\Cd)  \ar[d] _{\Cf}& &   (\Cx_{\Ct},\Cd_{\Ct}) \ar[d]_{\Cf_{\Ct}} &\\
				\Cs 	\ar@/_1pc/@{-->}[rrr] _\Phi &  \bar{\Cs}\ar[l]_{\tau} \ar[r]^{\rho}& 	 \Ct   \ar[r]^{\pi} &\Cv.
			}$$
			Let $\bar{U}\subset \Cs$ be an open subset such that $(\Cx,\Cd)\times_{\Cs} \bar{U}\cong (\Cx_{\Ct},\Cd_{\Ct})\times _{\Ct} \bar{U}$.
			
			Let $\bar{\Cs}_p$ be an irreducible component of $\bar{\Cs}$, let $\bar{\Ct}$ be normalization of the Stein factorization of $\rho:\bar{\Cs}_p\rightarrow \Ct$ so that $\bar{\Ct}\rightarrow \Ct$ is finite and $\bar{\Ct}$ is normal, and let $\tilde{\Cs}$ be the normalization of the main component of $\bar{\Cs}\times _\Ct \bar{\Ct}$. Repeating this argument for $\tilde{\Cs}\rightarrow \bar{\Ct}$ finitely many times, we may assume the general fiber of $\tilde{\Cs}\rightarrow \bar{\Ct}$ is irreducible. Then we replace $\Ct$ by $\bar{\Ct}$, $(\Cx_\Ct,\Cd_\Ct)$ by $(\Cx_\Ct,\Cd_\Ct)\times_{\Ct}\bar{\Ct}$, $\bar{\Cs}$ by $\tilde{\Cs}$, $\rho:\bar{\Cs}\rightarrow \Ct$ by $\tilde{\Cs}\rightarrow \bar{\Ct}$, $\tau:\bar{\Cs}\rightarrow \Cs$ by the composition of $\tilde{\Cs}\rightarrow \bar{\Cs}$ with $\tau:\bar{\Cs}\rightarrow \Cs$, $\pi:\Ct\rightarrow \Cv$ by the composition of $\bar{\Ct}\rightarrow \Ct$ with $\pi:\Ct\rightarrow \Cv$, and we assume that
			\begin{itemize}
				\item[(1)] the general fiber of $\bar{\Cs}\rightarrow \Ct$ is irreducible.
			\end{itemize}
			
			Let $\bar{\Cs}^o$ be an open subset of $\bar{U}$ and $\Cs^o\subset \Cs$ be an open subset such that 
			\begin{itemize}
				\item[(2)] $\tau$ is \'etale on $\bar{\Cs^o}$ and $\tau(\bar{\Cs}^o)=\Cs^o$,
				\item[(3)] $(\Cx,\Cd)\times_{\Cs} \bar{\Cs^o}\cong (\Cx_{\Ct},\Cd_{\Ct})\times _{\Ct} \bar{\Cs^o}$,
				\item[(4)] $\Phi$ is a smooth morphism on $\Cs^o$, and
				\item[(5)] $(\Cx,\Cd)\rightarrow \Cs$ has a fiberwise log resolution over $\Cs^o$.
			\end{itemize}
			By Lemma \ref{decompose to get weakly bounded}, we can choose an open subset $\Cv^o\subset \Cv$ and let $\Ct^o$ be the preimage of $\Cv^o$, such that
			\begin{itemize}
				\item[(6)] $\Cv^o$ is weakly bounded,
				\item[(7)] $\Ct^o= \pi^{-1}(\Cv^o)$,
				\item[(8)] $\Cv^o$ and $\Ct^o$ are smooth,
				\item[(9)] $\pi$ is an \'etale morphism on $\Ct^o$, and
				\item[(10)] $(\Cx_{\Ct},\Cd_{\Ct})\rightarrow \Ct$ has a fiberwise log resolution over $\Ct^o$.
			\end{itemize}
			Next we replace $\Cs^o$ by $\Phi^{-1}\Cv^o$ and $\bar{\Cs}^o$ by the $\rho^{-1}\Ct^o$, then
			\begin{itemize}
				\item[(11)] $\Phi^{-1}(\Cv^o)=\Cs^o$ and $\rho^{-1}(\Ct^o)=\bar{\Cs}^o$.
			\end{itemize}
			
			Notice that $(\Cx,\Cd)\times _\Cs (\Cs\setminus \Cs^o)\rightarrow (\Cs\setminus \Cs^o)$ is a again family of log Calabi--Yau pairs, by Theorem \ref{moduli part is nef and good} again we have a diagram, then we repeat the same argument and get a diagram for an open susbet of $\Cs\setminus \Cs^o$. Thus we have a stratification of $\Cs$. We replace $\Cs^o$ by this stratification and morphisms $\tau,\rho,\Phi,\pi$ by $\tau|_{\bar{\Cs}^o},\rho|_{\bar{\Cs}^o},\Phi|_{\Cs^o},\pi|_{\Ct^o}$, then we have the following diagram
			$$\xymatrix{
				(\Cx_{\Cs^o},\Cd_{\Cs^o})  \ar[d] _{\Cf^o}& &   (\Cx_{\Ct^o},\Cd_{\Ct^o}) \ar[d]_{\Cf_{\Ct^o}} &\\
				\Cs^o 	\ar@/_1pc/[rrr] _\Phi &  \bar{\Cs}^o \ar[l]_{\tau} \ar[r]^{\rho}& 	 \Ct^o   \ar[r]^{\pi} &\Cv^o,
			}$$
			which satisfies conditions (1)-(11).

			Define $\Cx_{{\bar{\Cs}^o}}:= \Cx_{\Cs^o}\times_{{\Cs^o}}{\bar{\Cs}^o}=\Cx_{\Ct^o}\times_{\Ct^o}\bar{\Cs}^o$. Let $\tau_\Cx,\rho_\Cx$ denote the natural projections $\Cx_{{\bar{\Cs}^o}}\rightarrow \Cx_{{\Cs^o}}$ and $\Cx_{{\bar{\Cs}^o}}\rightarrow \Cx_{\Ct^o}$. 
			
			$$\xymatrix{
				\Cx_{\Cs^o}  &  \Cx_{\bar{\Cs}^o} \ar[l]_{\tau_\Cx} \ar[r]^{\rho_\Cx} & \Cx_{\Ct^o}
			}$$

			Let $\{ \tilde{\Gamma}_{\Ct^o,\alpha},\alpha\in \mathscr{C}\}$ be a family of general sufficiently ample divisor on $\Cx_{\Ct^o}$ such that $\bigcap_{\alpha\in \mathscr{C}} \mathrm{Supp}(\tilde{\Gamma}_{\Ct^o,\alpha})=\emptyset$, where $\mathscr{C}$ is a variety of positive dimension. Define $\Gamma_{{\Cs^o},\alpha}:=(\tau_\Cx)_*\rho_\Cx^*\tilde{\Gamma}_{\Ct^o,\alpha}$, we claim the following: 
			\begin{claim}\label{claim 6.4} The divisor $\Gamma_{{\bar{\Cs}^o},\alpha}:=\tau_\Cx^*\Gamma_{{\Cs^o},\alpha}$ is equal to the pullback of a divisor $\Gamma_{\Ct^o,\alpha}$ on $\Cx_{\Ct^o}$ and $\bigcap_{\alpha\in \mathscr{C}} \mathrm{Supp}(\Gamma_{\Ct^o,\alpha})=\emptyset$.\end{claim}

			Recall by the construction, $\Cv^o$ is weakly bounded with respect to an ample divisor $\Ca$ on a projective compactification $\Cv$ of $\Cv^o$. We replace $\Ct$ by a smooth compactification of $\Ct^o$ such that $\Ct\setminus \Ct^o$ is an snc divisor and $\pi$ extends to a morphism $\Ct\rightarrow \Cv$. Then by the standard normal crossing assumptions, the moduli $\mathbf{b}$-divisor $\Nn$ of $(\Cx_{\Ct^o},\Cd_{\Ct^o})\rightarrow \Ct^o$ descends on $\Ct$. By Theorem \ref{moduli part is nef and good}, $\Nn_{\Ct}$ is nef and big. We fix a member in $|\Nn_{\Cs}|_{\mathbb{Q}}$ such that 
			\begin{itemize}
				\item $\Ct\setminus \Supp(\Nn_{\Ct})\subset \Ct^o$, and
				\item $\Supp(\Nn_{\Cs})\supset \Supp (\pi^*\Ca)$.
			\end{itemize}
			Then we fix an integer $l\gg 0$ such that $l\Nn_{\Ct}$ is Cartier and $l\Nn_{\Ct}-\pi^*\mathcal{A}\geq 0$.
			
			If $f:X_U\rightarrow U$ is a fibration, $\Delta_U$ is a $\mathbb{Q}$-divisor and $N_U$ is a divisor on $X_U$, such that the general fiber $X_g$ is klt and $(X_g,\Delta_g),N_g$ forms a $(d,c,v)$-polarized log Calabi--Yau pair, then by Remark \ref{Parametrizing space for log pairs}, after replacing $U$ by an open subset, $r(K_U+\Delta_U+tN_U)$ is relatively very ample and defines a closed embedding $g:X_U\hookrightarrow \mathbb{P}^n_U$. Thus, $f_U:(X_U,\Delta_U),N_U\rightarrow U$ together with $g$ is a strongly embedded $(d,c,v,t,r,\mathbb{P}^n)$-polarized log Calabi--Yau family over $U$.
			
			Because $(\Cx,\Cd)\rightarrow \Cs$ is the universal family, then there exists a morphism $\phi:U\rightarrow \Cs$ such that $(X,\Delta)\cong (\Cx,\Cd)\times_\Cs U$. Because $\Cs^o$ is a stratification of $\Cs$, then after replacing $U$ by an open subset, we may assume $\phi$ maps $U$ to $\Cs^o$, and we have $(X,\Delta)\cong (\Cx_{\Cs^o},\Cd_{\Cs^o})\times_{\Cs^o} U$.
			
			Let $Z$ be a smooth projective compactification of $U$ such that the moduli $\mathbf{b}$-divisor $\M$ of $(X_U,\Delta_U)\rightarrow U$ descends on $Z$ and $\Phi\circ \phi_U$ extends to a morphism $\phi_{\Cv}:Z\rightarrow \Cv$. Define $\bar{U}:=U\times _{\Cs^o}\bar{\Cs}^o$, because $\tau$ is a finite morphism, then $\bar{U}$ has a smooth compactification $\bar{Z}$ such that the finite cover $\bar{U}\rightarrow U$ extends to a generically finite cover $\tau_X:\bar{Z}\rightarrow Z$ and the composition of morphisms $\bar{U}\rightarrow \bar{\Cs}^o\rightarrow \Ct^o$ extends to a morphism $\phi_{\Ct}:\bar{Z}\rightarrow \Ct$.
			
			Because $(\Cx,\Cd)\times _\Cs \bar{\Cs}^o\cong (\Cx_{\Ct},\Cd_{\Ct})\times_\Ct \bar{\Cs}^o$, then the generic fiber of $(X_{\bar{U}},\Delta_{\bar{U}}):=(X_U,\Delta_U)\times _U \bar{U}\rightarrow \bar{U}$ is equal to the generic fiber of the pullback of $(X_{\Ct^o},\Cd_{\Ct^o})\rightarrow \Ct^o$ via $\phi_{\Ct}$. Since $(X_{\Ct^o},\Cd_{\Ct^o})\rightarrow \Ct^o$ has a fiberwise log resolution and $\phi_{\Ct}$ maps the generic point of $\bar{Z}$ into the generic point of $\Ct^o$, then by Theorem \ref{geneic point maps to log smooth locus means moduli divisor of restriction is restriction of moduli divisor}, we have $\bar{\M}_{\bar{Z}}=\phi_{\Ct}^*\Nn_{\Ct}$, where $\bar{\M}$ is the moduli $\mathbf{b}$-divisor of $(X_{\bar{U}},\Delta_{\bar{U}})\rightarrow \bar{U}$. By Proposition \ref{moduli part stable under base change}, we have $\bar{\M}_{\bar{Z}}=\tau_X^*\M_Z$. Recall that by the construction, we have $l\Nn_{\Ct}-\pi^*\Ca\geq 0$, then $l\bar{\M}_{\bar{Z}}-\tau_X^*\phi_{\Cv}^*\Ca\geq 0$ and $l\M_Z-\phi_{\Cv}^*\Ca\geq 0$.
		\end{proof}
	\end{thm}
	
	\begin{proof}[Proof of Claim \ref{claim 6.4}]
		
		Let $v\in \Cv^o$ be a closed point and $t\in \pi^{-1}v$ a closed point, $\Cx_t$ be the fiber of $\Cx_{\Ct}\rightarrow\Ct$ over $t$, define $\Cs_{t}:=\Phi^{-1}v$, $\bar{\Cs}_t:=\tau^{-1}\Cs_t$, $\Cx_{\Cs_t}\rightarrow \Cs_t$ and $\Cx_{\bar{\Cs}_t}\rightarrow \bar{\Cs}_t$ be the corresponding fiber product. Because $\Phi$ is the period map, by \cite[Proposition 2.1]{Amb05}, $\Cx_{\Cs_t }\rightarrow \Cs_t $ is an isotrivial fibration whose general fiber is isomorphic to $\Cx_t$. Also since $\Cx_{\bar{\Cs}}\cong \Cx_{\Ct}\times_{\Ct}\bar{\Cs}$, we have $\Cx_{\bar{\Cs}_t}=\Cx_t\times \bar{\Cs}_t$. 
		
		By the structure of the isotrivial fibration, there is a Galois cover $\tau_{t^\#}:\Cs^\#_t \rightarrow \Cs_t$ with Galois group $G$ which also acts on $\Cx_t$, such that $\Cx_{\Cs^\#_t}:=\Cx_{\Cs_t}\times _{\Cs_t}\Cs^\#_t \cong \Cx_t\times \Cs^\#_t$ and $\Cx_{\Cs_t}\cong \Cx_t\times \Cs^\#_t/G$, where $G$ acts diagonally on $\Cx_t\times \Cs^\#_t$. After replacing $\Cs_t^\#$ with a higher Galois cover, we may assume that there is a finite cover $\pi_{t^\#}:\Cs^\#_t \rightarrow \bar{\Cs}_t$ and $\tau_{t^\#}=\tau_t \circ \pi_{t^\#}:\bar{\Cs}_t\rightarrow \Cs_t$, where $\tau_t:=\tau|_{\bar{\Cs}_t}$.
		
		$$\xymatrix@R=1em{
			& &&    \Cx_{\Cs^\#_t} \ar[dd]  \ar[dl]^{\pi_{\Cx^\#_t}} \ar[dlll]_{\tau_{\Cx^\#_t}}  \ar[dr]^{\rho_{\Cx^\#_t}} &\\
			\Cx_{\Cs_t}\ar[dd]  &&     \Cx_{\bar{\Cs}_t} \ar[ll]^{\tau_{\Cx_{t}}} \ar[rr] _{\rho_{\Cx_t}} \ar[dd]    &    &  \Cx_t \ar[dd] \\
			&&&  \Cs^\#_t \ar[dl]^{\pi_{t^\#}} \ar[dlll]_{\tau_{t^\#}}  \ar[dr]^{\rho_{t^\#}}&\\
			\Cs_t	&&\bar{\Cs}_t \ar[ll]^{\tau_t} \ar[rr]_-{\rho_t}  & &   t\cong \mathrm{Spec}\ \mathbb{C}
		}$$
		
		Denote the projections $\Cx_{\Cs^\#_t}\rightarrow \Cx_{\Cs_t}$ and $\Cx_{\Cs^\#_t}\rightarrow \Cx_{\bar{\Cs}_t}$ by $\tau_{\Cx^\#_t}$ and $\pi_{\Cx^\#_t}$. Define a divisor $H_{t,\alpha}:= \tilde{\Gamma}_{\Ct^o,\alpha}|_{\Cx_t}$ on $\Cx_t$ and $H_{\#,\alpha}:=\sum_{g\in G} g^*\rho_{\Cx^\#_t}^*H_{t,\alpha}$ on $\Cx_{\Cs^\#_t}$, then $H_{\#,\alpha}$ is $G$-invariant. Because $G$ acts on $\Cx_t\times \Cs^\#_t$ diagonally and $\rho_{\Cx^\#_t}^*H_{t,\alpha}$ is vertical over $\Cx_t$, then $H_{\#,\alpha}$ is vertical over $\Cx_t$. 
		
		Because $\Cx_{\Cs_t}$ is the quotient of $ \Cx_{\Cs^\#_t}$ by $G$, there is a divisor $H_\alpha$ on $\Cx_{\Cs_t}$ such that $\tau_{\Cx^\#_t}^* H_\alpha=H_{\#,\alpha}$. 
		Also because $\tau_{\Cx^\#_t}= \tau_{\Cx_t}\circ \pi_{\Cx^\#_t}$, then $\tau_{\Cx_t}^* H_\alpha\leq (\pi_{\Cx^\#_t})_*H_{\#,\alpha}$. Since $H_{\#,\alpha}$ is vertical over $\Cx_t$, $(\pi_{\Cx^\#_t})_*H_{\#,\alpha}$ is also vertical over $\Cx_t$ via $\rho_{\Cx_t}$, and it means that $\tau_{\Cx_t}^* H_\alpha$ is also vertical over $\Cx_t$. It is easy to see that $\bigcap_{\alpha\in \mathscr{C}}H_\alpha=\emptyset$.
		
		Because $\rho_{\Cx^\#_t}^* H_{t,\alpha}\leq H_{\#,\alpha}=\tau_{\Cx^\#_t}^* H_\alpha$ and $\pi_{\Cx^\#_t}$ is a finite cover, then we have 
		$$\Supp(\tau_{\Cx_t}^*(\tau_{\Cx_t})_* \rho_{\Cx_t}^* H_{t,\alpha})\subset \Supp( \tau_{\Cx_t}^*H_\alpha).$$ 
		$\tau_{\Cx_t}^*H_\alpha$ is vertical over $\Cx_t$ implies that $ \tau_{\Cx_t}^*(\tau_{\Cx_t})_*\rho_{\Cx_t}^* H_{t,\alpha}$ is vertical over $\Cx_t$. 
		Because $\bigcap_{\alpha\in \mathscr{C}}H_\alpha=\emptyset$, $\bigcap_{\alpha\in \mathscr{C}}((\tau_{\Cx})_*\rho^*_{\bar{\Cx}}\tilde{\Gamma}_{\Ct,\alpha})|_{\Cx_{\Cs_t}}=\emptyset$. Because $t$ can be any closed point, then $\tau_{\Cx}^*(\tau_{\Cx})_*\rho^*_{\bar{\Cx}}\tilde{\Gamma}_{\Ct,\alpha}$ is vertical over $\Cx_{\Ct}$ and 
		$\bigcap_{\alpha\in \mathscr{C}}(\tau_{\Cx})_*\rho^*_{\bar{\Cx}}\tilde{\Gamma}_{\Ct,\alpha}=\emptyset$.
		
		Recall that the general fiber of $\rho:\bar{\Cs^o}\rightarrow \Ct^o$ is irreducible, then there exists an open subset $V\subset \Ct^o$ such that every fiber of $\rho_{\Cx}:\Cx_{\bar{\Cs}^o}\rightarrow \Cx_{\Ct^o}$ over $\mathcal{F}_{\Ct^o}^{-1}V$ is irreducible. Because $\tilde{\Gamma}_{\Ct,\alpha}$ is ample on $\Cx_{\Ct^o}$, then $\tilde{\Gamma}_{\Ct,\alpha}|_{\Cx_t}\neq 0$ for any closed point $t\in \Ct^o$. Also since $\tau_{\Cx}^*(\tau_{\Cx})_*\rho^*_{\bar{\Cx}}\tilde{\Gamma}_{\Ct,\alpha}$ is vertical over $\Cx_{\Ct}$, the generic point of its image on $\Cx_{\Ct}$ is contained in $\mathcal{F}_{\Ct^o}^{-1}V$, then there is a divisor $\Gamma_{\Ct^o,\alpha}$ on $\Cx_{\Ct^o}$ such that $\rho_\Cx^*\Gamma_{\Ct^o,\alpha}=\tau_{\Cx}^*(\tau_{\Cx})_*\rho^*_{\bar{\Cx}}\tilde{\Gamma}_{\Ct,\alpha}$. Since $\bigcap_{\alpha\in \mathscr{C}}(\tau_{\Cx})_*\rho^*_{\bar{\Cx}}\tilde{\Gamma}_{\Ct,\alpha}=\emptyset$, then we have $\bigcap_{\alpha\in \mathscr{C}}\Gamma_{\Ct^o,\alpha}=\emptyset$.
	\end{proof}

	The next theorem is about boundedness the base of a fibration, we show that for any log pair $(X,\Delta)\in \Cp(d,\Ci,v,V)$ toghther with a fibration $f:X\rightarrow Z$, $Z$ is birationally bounded.
	
	\begin{thm}\label{base is bounded}
		Let $d$ be a positive integer, $\Ci\subset [0,1]\cap\mathbb{Q}$ a DCC set and $v,V$ two positive rational numbers. There exists a positive rational number $c$ and a quasi-projective smooth morphism $\Cu \rightarrow T$, where $T$ is of finite type, such that:
		
		Let the following be the commutative diagram corresponding to the numbers $d,c,v$ defined in Theorem \ref{essential diagram} 
		$$\xymatrix{
			(\Cx_{\Cs^o},\Cd_{\Cs^o})  \ar[d] _{\Cf}& (\Cx_{\bar{\Cs^o}},\Cd_{\bar{\Cs^o}})\ar[l]_{\tau_\Cx} \ar[r]^{\rho_{\Cx}} \ar[d] &   (\Cx_{{\Ct^o}},\Cd_{{\Ct^o}}) \ar[d]_{\Cf_{{\Ct^o}}} &\\
			\Cs^o 	\ar@/_1pc/[rrr] _\Phi &  \bar{\Cs^o}\ar[l]_{\tau} \ar[r]^{\rho}& 	 {\Ct^o}   \ar[r]^{\pi} &\Cv^o
		}$$
		and $\Cv^o\hookrightarrow \Cv$ be the corresponding closure. Let $(X,\Delta)\in \Cp(d,\Ci,v,V)$ be a log pair together with the divisor $A$, the morphism $f:X\rightarrow Z$ and $K_X+\Delta\sim_{\mathbb{Q}}f^*(K_Z+B_Z+\M_Z)$. 
		Then a general fiber $(X_g,\Delta_g),A_g$ is a $(d,c,v)$-polarized log Calabi--Yau pair, and there exists an open subset $U\subset Z$, a morphism $\phi:U\rightarrow \Cs^o$ and a closed point $t\in T$ such that
		\begin{enumerate}
			\item $(X,\Delta)\times_Z U\cong (\Cx_{\Cs^o},\Cd_{\Cs^o})\times_{\Cs^o} U$,
			\item $Z\setminus \Supp(B_Z)$ has an open subset which is isomorphic to an big open subset of $\Cu_t$, and
			\item $\Phi\circ \phi:U\rightarrow \Cv^o$ extends to a morphism $\phi_{\Cv}:\Cu_t \rightarrow \Cv^o$.
		\end{enumerate}
		
	\begin{proof}
		Recall that a big open subset $U$ of a normal $V$ is an open subset and $\mathrm{codim}(V\setminus U)\geq 2$.
		
		Fix $(X,\Delta)\in \Cp(d,\Ci,v,V)$ toghther with a fibration $f:X\rightarrow Z$, $K_X+\Delta\sim_{\mathbb{Q}}f^*(K_Z+B_Z+\M_Z)$ and divisors $A$ and $H$.
		
		\textit{Step 1}. In this step we show that there exists a projective smooth morphism $\Cw\rightarrow T$, where $T$ of finite type, depending only on $d,\Ci,v,V$ such that $Z$ is birationally equivalent to $\Cw_t$ for a closed point $t\in T$.

		Because $|H|$ defines a birational map $h: Z\dashrightarrow W$, let $p: Z'\rightarrow Z,q: Z'\rightarrow W$ be a common resolution, then there is a very ample divisor $H_W$ on $W$ and an effective divisor $F$ on $Z'$ such that
		\begin{equation}\label{birational contraction}
			p^*H=q^*H_W+F.
		\end{equation}
		By the definition of $\Cp(d,\Ci,v,V)$, $K_Z+B_Z+\M_Z$ is nef, then $\vol(H_W)\leq \vol(H)\leq \vol(H+K_Z+B_Z+\M_Z)\leq V$. Then by the boundedness of the Chow varieties, $W$ is in a bounded family. We denote this bounded family by $\Cw\rightarrow T$ and the natural relative very ample divisor by $H$, suppose $W\cong \Cw_t$ and $H_W=H_{\Cw_t}$. 
		
		After replacing $\Cw$ by a log resolution of the generic fiber of $\Cw\rightarrow T$ and passing to a stratification of $T$, we may assume that $\Cw\rightarrow T$ is a smooth morphism. Because $H$ is big, we can replace $H_{\Cw}$ by a very ample divisor on the new family, $H$ by a higher multiple and $V$ by a larger constant so that equation \eqref{birational contraction} still holds.
		
		\vspace{0.4cm}
		
		\textit{Step 2}. In this step we show that a general fiber $(X_g,\Delta_g),A_g$ is a $(d,c,v)$-polarized log Calabi--Yau pair, then prove (1).
		
		Because the general number $(X_g,\Delta_g)$ is a log Calabi--Yau pair and $\mathrm{coeff}(\Delta_g)\subset \Ci$ is a DCC set, then by \cite[Theorem 1.5]{HMX14}, $\mathrm{coeff}(\Delta_g)$ is in a finite set. Choose a positive rational number $c$ such that $\mathrm{coeff}(\Delta)_g\subset c\mathbb{N}$. It is easy to see that a general fiber $(X_g,\Delta_g), A_g$ of $f$ is a $(d,c,v)$-polarized log Calabi--Yau pair. By Theorem \ref{essential diagram} and the construction of the morphism $(\Cx_{\Cs^o},\Cd_{\Cs^o})\rightarrow \Cs^o$, there exists an open subset $U\subset Z$ such that $(X,\Delta)\times_Z U$ is equal to the pullback of $(\Cx_{\Cs^o},\Cd_{\Cs^o})\rightarrow \Cs^o$ via the moduli map $\phi:U\rightarrow \Cs^o$. We define $\phi_\Cv:=\Phi\circ \phi:U\rightarrow \Cv^o$, note that because $U$ is birationally equivalent to $\Cw_t$, $\phi_\Cv$ also defines a map $\Cw_t\dashrightarrow \Cv$.

		\vspace{0.4cm}
		
		\textit{Step 3}. In this step we construct a birational morphism $g_c:Z_c\rightarrow \Cw_t$ such that $\phi_\Cv$ extends to a morphism $\phi_\Cv:Z_c\rightarrow \Cv$ and there is an ample divisor $A_{Z_c}$ on $Z_c$ whose degree with respect to $g_c^*H_{\Cw_t}$ is bounded from above. We define $\Supp(g_{c*}A_{Z_c})$ to be the boundary part $\Cd_t$ on $\Cw_t$, then we show $\Cw_t\setminus \Cd_t$ has an big open subset which is isomorphic to an open subset of $Z\setminus \Supp(B_Z)$.
		
		Let $Z'$ be a log resolution of $(Z, B_Z)$ such that $\phi_\Cv$ extend to a morphism $Z'\rightarrow \Cv$, $\M$ descends on $Z'$ and there is a projective birational morphism $g: Z'\rightarrow\Cw_t$. 
		Define $B_{Z'}$ to be the strict transform of $B_Z$ plus the exceptional divisors of $Z'\rightarrow Z$, because $(Z',B_{Z'})$ is log smooth, $(Z',B_{Z'})$ is a dlt pair. Define $B'_{Z'}=\frac{1}{2}\mathrm{red}(B_{Z'})$. By the ACC for log canonical thresholds and the construction of the boundary part, $\mathrm{coeff}(B_{Z'})$ is in a DCC set $\Ci'$ depending only on $d,\Ci$. In particular, there is a positive rational number $\delta$ depending only on $d,\Ci$ such that $\mathrm{coeff}(B_{Z'})\geq \delta$, then $\delta B'_{Z'}\leq B_{Z'}$. 
		
		Suppose $\mathrm{dim}Z=n$. By the length of extremal rays, $K_{\Cw_t}+3nH_{\Cw_t}$ is ample. Also because $Z'$ and $\Cw_t$ are smooth, $K_{Z'}\geq g^*K_{\Cw_t}$, then $K_{Z'}+3ng^*H_{\Cw_t}$ is big. Let ${Z'}\dashrightarrow Z_c$ be the canonical model of $K_{Z'}+B'_{Z'}+3ng^*H_{\Cw_t}+3n\phi_\Cv^*\Ca$, where $\Ca$ is the very ample divisor on $\Cv$ defined in Theorem \ref{essential diagram}. Here we consider $({Z'},B'_{Z'}+3ng^*H_{\Cw_t}+3n\phi_\Cv^*\Ca)$ as a generalized log pair with nef part $3ng^*H_{\Cw_t}+3n\phi_\Cv^*\Ca$.
		$$\xymatrix{
			& {Z'} \ar[dl]_g \ar@{-->}[d] \ar[dr]^{\phi_\Cv}&   \\
			\Cw_t	&  Z_c\ar[l]^{g_c} \ar[r]_{\phi_{\Cv_c}}& 	 \Cv.
		} $$
		By \cite[Lemma 4.4]{BZ16}, the contraction ${Z'}\dashrightarrow Z_c$ is $g^*H_{\Cw_t}$ and $\phi_\Cv^*\Ca$-trivial, so there are two morphisms $g_c:Z_c\rightarrow \Cw_t$ and ${\phi_{\Cv_c}}:Z_c\rightarrow \Cv$. 
		
		Because $B'_{Z_c}$ is the pushforward of $B'_{Z'}$ and $\mathrm{coeff}(B'_{Z'})=\frac{1}{2}$, by \cite[Theorem 8.1]{BZ16}, there is a rational number $e\in (0,1)$ depending only on $\mathrm{dim}Z=n$ such that $K_{Z_c}+eB'_{Z_c}+3ng_c^*H_{\Cw_t}+3n\phi_{\Cv_c}^*\Ca$ is a big $\mathbb{Q}$-divisor. Also because the $\mathrm{coeff}(B'_{Z'})$ is equal to $\frac{1}{2}$, $\mathrm{coeff}(eB'_{Z'})$ is equal to $\frac{e}{2}$, then by \cite[Theorem 1.3]{BZ16}, there exists an integer $m\in \mathbb{N}$ depending only on $n$ such that $|m(K_{Z_c}+eB'_{Z_c}+3ng_c^*H_{\Cw_t}+3n\phi_{\Cv_c}^*\Ca)|$ defines a birational map. We choose $m$ sufficiently divisible such that both $m\frac{1-e}{2},m\frac{e}{2}\in \mathbb{N}$, then there is an effective divisor
		$$A'_{Z_c}\sim m(K_{Z_c}+eB'_{Z_c}+3ng_c^*H_{\Cw_t}+3ne\phi_{\Cv_c}^*\Ca).$$
		Define $A_{Z_c}:=A'_{Z_c}+m(1-e)B'_{Z_c}+3nm(1-e)\phi_{\Cv_c}^*\Ca$. Because $A'_{Z_c},B'_{Z_c},\phi_{\Cv_c}^*\Ca$ are effective and $Z_c$ is the canonical model of $K_{Z'}+B'_{Z'}+3ng^*H_{\Cw_t}+3n\phi_\Cv^*\Ca$, then
		$$A_{Z_c}\sim m(K_{Z_c}+B'_{Z_c}+3ng_c^*H_{\Cw_t}+3n\phi_{\Cv_c}^*\Ca)$$
		is an effective ample divisor and $\Supp(B'_{Z_c}+\phi_{\Cv_c}^*\Ca)\subset \Supp(A_{Z_c})$.
		
		Write $N:=2(2n+1)g_c^*H_{\Cw_t}$. By \cite[Lemma 3.2]{HMX13}, 
		\begin{equation}
			\begin{aligned}
				&\mathrm{red}(A_{Z_c}).N^{n-1}\\ 
				\leq & 2^n\vol(K_{Z_c}+\mathrm{red}(A_{Z_c})+N)\\
				\leq & 2^n\vol(K_{Z_c}+m(K_{Z_c}+B'_{Z_c}+3ng_c^*H_{\Cw_t}+3n\phi_{\Cv_c}^*\Ca)+N)\\
				\leq & 2^n\vol((m+1)K_{Z_c}+mB'_{Z_c}+(3nm+2(2n+1))g_c^*H_{\Cw_t}+ 3nm\phi_{\Cv_c}^*\Ca)\\
				\leq & (2a)^n\vol(K_{Z_c}+B'_{Z_c}+3ng_c^*H_{\Cw_t}+3n\phi_{\Cv_c}^*\Ca),
			\end{aligned}
		\end{equation}
		where $a=\mathrm{max}\{m+1,m+\frac{2(2n+1)}{m}\}$.
		
		Since $\M$ descends on $Z_c$, then by Theorem \ref{essential diagram}, we have $l\M_{Z'}\geq \phi_\Cv^* \Ca$.
		Also because ${Z'}\dashrightarrow Z_c$ is the canonical model of $K_{{Z'}}+B'_{{Z'}}+3ng^*H_{\Cw_t}+3n\phi_\Cv^*\Ca$, then 
		\begin{equation}
			\begin{aligned}
				&  \vol(K_{Z_c}+B'_{Z_c}+3ng_c^*H_{\Cw_t}+3n\phi_{\Cv_c}^*\Ca)\\
				=&  \vol(K_{Z'}+B'_{Z'}+3ng^*H_{\Cw_t}+3n\phi_\Cv^*\Ca)\\
				\leq & \vol(K_{Z'}+B'_{Z'}+3ng^*H_{\Cw_t}+3nl\M_{Z'}).
			\end{aligned}
		\end{equation}
		Consider the following equation
		\begin{equation*}
			\begin{aligned}
				&K_{Z'}+\delta B'_{Z'}+3ng^*H_{\Cw_t}+3nl\M_{Z'}\\
				=&\delta(K_{Z'}+B'_{Z'}+3ng^*H_{\Cw_t}+3nl\M_{Z'})+(1-\delta)(K_{Z'}+3ng^*H_{\Cw_t}+3nl\M_{Z'}).
			\end{aligned}
		\end{equation*}
		Because $K_{Z'}+3ng^*H_{\Cw_t}+3nl\M_{Z'}$ is big and $\delta B'_{Z'}\leq B_{Z'}$, we have
		\begin{equation}
			\begin{aligned}
				&  \vol(K_{Z'}+B'_{Z'}+3ng^*H_{\Cw_t}+3nl\M_{Z'})\\
				\leq &  (\frac{1}{\delta})^n\vol(K_{Z'}+\delta B'_{Z'}+3ng^*H_{\Cw_t}+3nl\M_{Z'})\\
				\leq &  (\frac{1}{\delta})^n\vol(K_{Z'}+B_{Z'}+3ng^*H_{\Cw_t}+3nl\M_{Z'}).
			\end{aligned}
		\end{equation}
		By assumption, $K_{Z'}+B_{Z'}+\M_{Z'}+g^*H_{\Cw_t}$ is big and $\mathrm{coeff}(B_{Z'})$ is in the DCC set $\Ci'$ depending only on $d,\Ci$, then by \cite[Theorem 8.1]{BZ16}, there is a positive rational number $e'<1$ depending only on $d,\Ci$ such that $K_{Z'}+B_{Z'}+e'\M_{Z'}+e'g^*H_{\Cw_t}$ is big. Consider the following inequality,
		\begin{equation*}
			\begin{aligned}
				&\alpha(K_{Z'}+B_{Z'}+3ng^*H_{\Cw_t}+3nl\M_{Z'}) +(1-\alpha)(K_{Z'}+B_{Z'}+e'g^*H_{\Cw_t}+e'\M_{Z'})\\
				\leq &K_{Z'}+B_{Z'}+\M_{Z'}+g^*H_{\Cw_t},
			\end{aligned}
		\end{equation*}
		where $\alpha=\frac{1-e'}{3nl-e'}$. Then
		\begin{equation*}
			\begin{aligned}
				&  \vol(K_{Z'}+B_{Z'}+3ng^*H_{\Cw_t}+3nl\M_{Z'})\\
				\leq & (\frac{1}{\alpha})^n\vol(K_{Z'}+B_{Z'}+\M_{Z'}+H)\\
				\leq & (\frac{1}{\alpha})^nV,
			\end{aligned}
		\end{equation*}
		where the last inequality comes from the assumptions on $\Cp(d,\Ci,v,V)$. Then, we have
		$$\mathrm{red}(A_{Z_c}).N^{n-1}\leq (\frac{2a}{\alpha\delta})^nV.$$ 
		By the boundedness of Chow variety, $(\Cw_t,\Supp((g_c)_*A_{Z_c}))$ is in a log bounded family. We denote this family of log pairs by $(\Cw,\Cd)\rightarrow T$, where $T$ is of finite type and $\Cd$ is reduced. 
		
		By the construction, $\Supp(B'_{Z'})$ contains the strict transform of $B_Z$ on $Z'$ plus the exceptional divisor of $Z'\rightarrow Z$, then $\Supp(B'_{Z_c})$ contains the strict transform of $B_Z$ on $Z_c$ plus exceptional divisor of $Z_c\dashrightarrow Z$. Recall that $\Supp(B'_{Z_c})\subset \Supp(A_{Z_c})$, then $\Supp((g_c)_*A_{Z_c})=\Cd_t$ contains the strict transform of $B_Z$ on $\Cw_t$ plus exceptional divisor of $\Cw_t\dashrightarrow Z$, then $\Cw_t\setminus \Cd_t$ has a big open subset which is isomorphic to an open subset of $Z\setminus \Supp(B_Z)$.
		
		\vspace{0.4cm}
		
		\textit{Step 4}. In this step we show that $\Phi\circ \phi $ extends to a morphism on $\Cw_t\setminus \Cd_t$ and it maps $\Cw_t\setminus \Cd_t$ to $\Cv^o$. We define $\Cu:=\Cw\setminus \Cd$, then (2) and (3) hold.

		Because $A_{Z_c}$ is effective and ample, $\Cw_t$ is smooth, by the negativity lemma, $g_c^* (g_c)_*A_{Z_c}-A_c$ is effective and contains all $g_c$-exceptional divisor. Define $\Cu:=\Cw\setminus \Cd$, then $\Cu_t=\Cw_t\setminus \Supp((g_c)_*A_{Z_c})$. Because $\Cw_t$ is smooth, the exceptional locus of $g_c$ is pure of codimension 1, then
		$$\Cu_t\subset Z_c\setminus \Supp(A_{Z_c}),$$
		and $\phi_{\Cv_c}:Z_c\rightarrow \Cv$ induce a morphism $\phi_{\Cv_c}:\Cu_t\rightarrow \Cv$. Because $\Supp(\phi_{\Cv_c}^*\Ca)\subset \Supp(A_{Z_c})$ and $\Cv^o\subset \Cv\setminus \Supp(\Ca)$, then $\phi_{\Cv_c}(\Cu_t)\subset \Cv^o$. 
	\end{proof}
	\end{thm}
	
	The following theorem shows that for any $(X,\Delta)\in \Cp(d,\Ci,v,V)$, the corresponding moduli map can be parametrized by a morphism on a scheme of finite type.
	
	\begin{thm}\label{parametrize moduli map and get universal ambro models}
		Let $\Cu \rightarrow T'$ be a quasi-projective smooth morphism, where $T'$ is of finite type. Let 
		$$\xymatrix{
			(\Cy_{\Ct^o},\Cr_{\Ct^o}) \ar[d] &\\
			\Ct^o   \ar[r]^{\pi} &\Cv^o.
		}$$
		be a commutative diagram, where
		\begin{itemize}
			\item $\pi$ is \'etale,
			\item $\Cv^o$ is weakly bounded, and
			\item $(\Cy_{\Ct^o},\Cr_{\Ct^o})$ is log smooth over $\Ct^o$.
		\end{itemize}
	
		Then there exists a projective log smooth morphism $(\Cw,\Cd)\rightarrow T$, where $T$ is of finite type and $\Cd$ is reduced, and a morphism $\Theta:\Cw\setminus \Cd \rightarrow \Cv^o$ such that:

		Suppose ${t'}\in T'$ is a closed point and $\phi_\Cv:\Cu_{t'}\rightarrow \Cv^o$ is a morphism, define $\bar{U}:=\Cu_{t'}\times _{\Cv^o}\Ct^o$ and $(Y_{\bar{U}},R_{\bar{U}}):=(\Cy_{\Ct^o},\Cr_{\Ct^o})\times _{\Ct^o}\bar{U}$. Suppose $U\hookrightarrow \Cu_{t'}$ is an open subset, $(X_U,\Delta_U)\rightarrow U$ is an lc-trivial fibration, and there exists $V$, two quasi-finite dominate morphisms $V\rightarrow U$ and $V\rightarrow \bar{U}$ such that the generic fiber of $(Y_{\bar{U}},\Supp(R_{\bar{U}}))\times _{\bar{U}} V\rightarrow V$ is a log resolution of the generic fiber of $(X_U,\Supp(\Delta_U))\times_ U V\rightarrow V$. Denote the moduli $\mathbf{b}$-divisor of $(X_U,\Delta_U)\rightarrow U$ by $\M$.
		Then there exists a closed point $t\in T$ such that
		\begin{enumerate}
			\item $\Cw_t\setminus \Cd_t$ is isomorphic to $\Cu_{t'}$, 
			\item $\M$ descends on $\Cw_t$, and
			\item $\phi_{\Cv}:=\Theta|_{\Cw_t\setminus \Cd_t}$.
		\end{enumerate}
		\begin{proof}
			Because $\Cv^o$ is weakly bounded, by Theorem \ref{weakly bounded morphisms are bounded}, there is a scheme of finite type $\mathscr{W}$ and a morphism $\Theta':\mathscr{W}\times \Cu \rightarrow \Cv^o$ depending only on $\Cu\rightarrow T'$ and $\Cv^o$, such that $\phi_{\Cv}=\Theta'|_{\{p\}\times \Cu_{t'}}$ for a closed point $p\in \mathscr{W}.$
			
			Define $T:=\mathscr{W}\times T'$, we replace $\Cu$ by $\mathscr{W}\times \Cu$ and define $\bar{\Cu}:=\Cu\times _{\Cv^o}\Ct^o$, where the morphism $\Cu\rightarrow \Cv^o$ is $\Theta'$. Let $t:=\{p\}\times t' \in T$, since $\phi_{\Cv}:=\Theta'|_{\{p\}\times \Cu_{t'}}$, we have $\bar{U}:=\bar{\Cu}_t$.
			
			Since $\Ct^o\rightarrow \Cv^o$ is \'etale, $\Cu\rightarrow T$ is smooth, then $\bar{\Cu}\rightarrow T$ is smooth and $\bar{\Cu}\rightarrow \Cu$ is \'etale. Let $K(\tilde{\Cu})/K(\Cu)$ be the Galois closure of $K(\bar{\Cu})/K(\Cu)$ and $\tilde{\Cu}\rightarrow \Cu$ be the \'etale Galois cover with group $G$. It is easy to see that every fiber of $\tilde{\Cu}\rightarrow T$ is smooth.

			Let $\tilde{\Cu}\hookrightarrow \tilde{\Cw}'/T$ be a compactification over $T$, $\tilde{\Cw}\rightarrow \tilde{\Cw}'$ be a $G$-equivariant log resolution of $(\tilde{\Cw}',\tilde{\Cw}'\setminus \tilde{\Cu})$ and $\Cw$ be the quotient of $\tilde{\Cw}$ by $G$. Because $\tilde{\Cw}$ is a compactification of $\tilde{\Cu}$ over $T$ and the quotient of $\tilde{\Cu}$ by $G$ is $\Cu$, then $\Cw$ is a compactification of $\Cu$ over $T$. Define $\Cd:=\Cw\setminus \Cu$, then $\Theta$ induces a morphism on $\Cw\setminus \Cd\rightarrow T$ which clearly satisfies (3). Because (1) automatically holds by the construction of $(\Cw,\Cd)$, we only need show that $(\Cw,\Cd)\rightarrow T$ satisfies (2).

			Write $(\Cy_{\tilde{\Cu}_t},\Cr_{\tilde{\Cu}_t}):=(\Cy_{\Ct^o},\Cr_{\Ct^o})\times _{\Ct^o}\tilde{\Cu}_t$, where the morphism $\tilde{\Cu}_t\rightarrow \Ct^o$ is the composition $\tilde{\Cu}_t\rightarrow \bar{\Cu}_t\rightarrow \Ct^o$. Since $\Cw_t\setminus \Cd_t$ is isomorphic to $\Cu_{t'}$, then $(\Cy_{\tilde{\Cu}_t},\Supp(\Cr_{\tilde{\Cu}_t}))$ is log smooth over $ \tilde{\Cu}_t$. Because $(\tilde{\Cw}_t,\tilde{\Cw}_t\setminus \tilde{\Cu}_t)$ is log smooth, by Theorem \ref{canonical bundle formula}.(d), the moduli $\mathbf{b}$-divisor of $(\Cy_{\tilde{\Cu}_t},\Cr_{\tilde{\Cu}_t})\rightarrow \tilde{\Cu}_t$ descends on $\tilde{\Cw}_t$, denote it by $\tilde{\M}$. 
			
			Let $(X_U,\Delta_U)\rightarrow U$ be an lc-trivial fibration that satisfies the conditions.
			Because $V\rightarrow \bar{\Cu}_t\rightarrow \Cu_t$ is a quasi-finite dominate morphism and $\tilde{\Cu}_t\rightarrow \Cu_t$ is finite, then we can replace $V$ by a finite cover such that there is a projective compactification $V\hookrightarrow W$ and a finite morphism $ W\rightarrow \tilde{\Cw_t}$. By assumption, the generic fiber of $(X,\Delta)\times_{\Cw_t} V\rightarrow V$ is crepant birationally equivalent to the generic fiber of $(\Cy_{\Ct^o},\Cr_{\Ct^o})\times_{\Ct^o}V\rightarrow V$, then the generic fiber of $(X,\Delta)\times_{\Cw_t} W\rightarrow W$ is crepant birationally equivalent to the generic fiber of $(\Cy_{\tilde{\Cu}_t},\Cr_{\tilde{\Cu}_t})\times_{\tilde{\Cu}_t}W\rightarrow W$.
			Note that the moduli part only depends on the crepant birational equivalent class of generic fiber. By Proposition \ref{moduli part stable under base change}, because the moduli $\mathbf{b}$-divisor $\tilde{\M}$ of $(\Cy_{\tilde{\Cu}_t},\Cr_{\tilde{\Cu}_t})\rightarrow \tilde{\Cu}_t$ descends on $\tilde{\Cw}_t$, then the moduli part of $(X,\Delta)\times_{\Cw_t} V\rightarrow V$ descends on $W$ and $\M$ descends on $\Cw_t$.
		\end{proof}
	\end{thm}
	
	The moduli map of an isotrivial fibration $(X,\Delta)\rightarrow Z$ is to map $Z$ into a closed point in the moduli space corresponding to a general fiber $(X_g,\Delta_g)$, which is exactly the same with a trivial fibration. Therefore, one can not distinguish different fibrations just through moduli maps, not even up to birational equivalence. The next theorem is used to compare volumes of different fibrations with the same moduli map.
	
	\begin{thm}\label{compare volumes}
		Let $(W,D),(W',D')$ be two projective log smooth pairs with reduced boundaries, $\tau:W'\rightarrow W$ a generically finite morphism such that $\tau^{-1}(W\setminus D)=W'\setminus D'$ and $\tau$ is finite on $W'\setminus D'$.
		Suppose $(Y,R)\rightarrow W$ is an lc-trivial fibration,
		$Y'\rightarrow W'$ is a fibration, and $R'$ is a $\mathbb{Q}$-divisor on $Y'$ such that 
		\begin{itemize}
			\item $Y$ is smooth,
			\item the generic fiber of $(Y,R)\rightarrow W$ is sub-klt and log smooth,
			\item $R'$ is horizontal over $W'$,
			\item $(Y',R' +\mathrm{red}(f_{Y'}^*D'))$ is log smooth and sub-lc,
			\item $(Y',R')$ is log smooth over $W'\setminus D'$, and
			\item the generic fiber of $(Y',R')\rightarrow W'$ is crepant birationally equivalent to the generic fiber of $(Y,R)\rightarrow W$ after a finite base change.
		\end{itemize}
		$(Y',R')\rightarrow W'$ defines an lc-trivial fibration over an open subset of $W'$. Denote the moduli $\mathbf{b}$-divisor of $(Y,R)\rightarrow W$ (respectively $(Y',R')\rightarrow W'$) by $\M$ (respectively $\M'$).
		
		Consider the following diagram 
		$$\xymatrix@R=0.3em@C=1.8em{
			\tilde{Y} \ar[rr]^{h} \ar[dr]_{f_{\tilde{Y}}} \ar[dddd]_{\tau_{Y}}&& \bar{Y} \ar[dl]^{f_{\bar{Y}}} \ar[dd]^{\pi_Y}\\
			&\bar{W} \ar[dd]^{\pi}&\\
			& & Y'\ar[dl]^{f_{Y'}}\\
			& W' \ar[dd]^{\tau}&\\
			Y \ar[dr]_{f_Y}&& \\
			& W, &
		}$$
		where 
		\begin{itemize}
			\item every variety is projective and normal,
			\item $\pi$ is finite,
			\item $\M$ descends on $W$ and $\M'$ descends on $W'$,
			\item $\bar{Y}$ the the normalization of the main component of $\bar{W}\times _{W'}Y'$,
			\item after base changing to $\bar{W}$, the generic fiber of $(Y',R')\rightarrow W'$ is crepant birationally equivalent to the generic fiber of $(Y,R)\rightarrow W$, and
			\item $\tilde{Y}$ is a common resolution of $\bar{Y}$ and $Y\times _W \bar{W}$.
		\end{itemize}
		Let $\Delta_Y$ be the $\mathbb{Q}$-divisor on $Y$ such that $K_Y+\Delta_Y\sim_{\mathbb{Q}}f_Y^*(K_W+D+\M_W)$ and $\Delta_Y|_{Y_\eta}=R|_{Y_\eta}$ for the general fiber $Y_\eta$ of $f_Y$. If 
		\begin{itemize}
			\item $(Y,\Delta_{Y,>0})$ is lc, and
			\item every codimension 1 point in $W\setminus D$ is dominated by a codimension 1 point in $Y\setminus\Supp(\Delta_{Y,<0})$,
		\end{itemize}
		then we have
		$$h_* \tau_Y^* (K_Y+\Delta_{Y,>0}) \leq \pi_Y^*(K_{Y'}+R'_{>0} +\mathrm{red}(f_{Y'}^*D')).$$
		\begin{proof}
			\textit{Step 1}. In this step we show the following inequality
			\begin{equation}\label{very hard equation}
				h_* \tau_Y^* (K_Y+\Delta_Y) \leq \pi_Y^*(K_{Y'}+R' +\mathrm{red}(f_{Y'}^*D')),
			\end{equation}
			then to prove the result, we only need to show that
			$$h_* \tau_Y^*\Delta_{Y,<0}\leq \pi_Y^*R'_{<0}.$$
			
			Because $(W',D')$ is log smooth and $\M'$ descends on $W'$, then $(W',D'+\M'_{W'})$ is a generalized lc pair. Also because $(Y',R')$ is log smooth over $W'\setminus D'$ and $R'$ is horizontal over $W'$, by Lemma \ref{inverse of canonical bundle formula}, there is a $\mathbb{Q}$-divisor $B'\leq \mathrm{red}(f_{Y'}^*D')$ on $Y'$ such that 
			$K_{Y'}+R'+B' \sim_{\mathbb{Q}} f_{Y'}^* (K_{W'}+D'+\M'_{W'})$. Then we have 
			$$f_{Y'}^* (K_{W'}+D'+\M'_{W'})\leq K_{Y'}+R'+\mathrm{red}(f_{Y'}^*D') $$
			
			Since $\pi^{-1}(W\setminus D)=W'\setminus D'$, $D'$ dominates $D$. Because both $(\Cw_t,\Cd_t)$ and $(\Cw'_t,\Cd'_t)$ are log smooth, by Hurwitz's formula (\cite[2.41.4]{Kol13}), we have
			$\tau^*(K_{W}+\Cd_t)\leq  K_{\Cw'_t}+\Cd'_t.$
			
			Since the generic fiber of $(Y',R')\rightarrow W'$ is crepant birationally equivalent to the generic fiber of $(Y,R)\rightarrow W$ after a finite base change, by Proposition \ref{moduli part stable under base change}, $\tau^* \M_{W} =\M'_{W'}$, then we have
			$$\tau^*(K_{W}+D+\M_W)\leq  K_{W'}+D'+\M'_{W'}.$$
			Also because $K_Y+\Delta_Y\sim_{\mathbb{Q}} f_Y^*(K_{W}+D+\M_W)$, we have
			$$h_* \tau_Y^* (K_Y+\Delta_Y) \leq \pi_Y^*(K_{Y'}+R' +\mathrm{red}(f_{Y'}^*D')).$$
			
			\vspace{0.3cm}
			
			\textit{Step 2}. In this step we decompose $h_* \tau_Y^*\Delta_{Y,<0}$ into several parts depending on its image on $\bar{W}$.
			
			Write $h_* \tau_Y^*\Delta_{Y,<0}=D_1+D_2+D_3+D_4$, where 
			\begin{itemize}
				\item every irreducible component of $D_1$ is horizontal over $\bar{W}$,
				\item every irreducible component of $D_2$ maps into a codimension $\geq 2$ subset in $\bar{W}$,
				\item every irreducible component of $D_3$ dominates an irreducible component of $\Supp (\pi^*D')$, and
				\item every irreducible component of $D_4$ dominates a prime divisor not contained in $\Supp (\pi^*D')$.
			\end{itemize}
			
			\vspace{0.3cm}
			
			\textit{Step 3}. In this step we show that $D_1=\pi_Y^*R'_{<0}$.
			
			Recall that $\pi$ is a finite morphism and $\bar{Y}$ the the normalization of the main component of $\bar{W}\times _{W'}Y'$, then $\pi_Y$ is a finite morphism. 
			Because $R'$ is horizontal over $W'$, then $R'_{<0}$ is horizontal over $W'$, $\pi^*_YR'_{<0}$ is horizontal over $\bar{W}$. Also because the generic fiber of $(Y,\Delta_Y)\rightarrow W$ is crepant birationally equivalent to the pullback of the generic fiber of $(Y',R')\rightarrow W'$, we have $D_1=\pi_Y^*R'_{<0}$. Adding both sides to Equation \eqref{very hard equation}, we have 
			\begin{equation}\label{D_1}
				h_* \tau_Y^* (K_Y+\Delta_Y)+D_1\leq \pi_Y^*(K_{Y'}+R'_{>0} +\mathrm{red}(f_{Y'}^*D')).
			\end{equation}
			
			\vspace{0.3cm}
			
			\textit{Step 4}. In this step we show that $h_* \tau_Y^* (K_Y+\Delta_Y)+D_2+D_3\leq \pi_Y^*(K_{Y'}+R'_{>0} +\mathrm{red}(f_{Y'}^*D'))$.

			Recall that $(Y',R')$ is log smooth over $W'\setminus D'$ and $\bar{Y}$ the the normalization of the main component of $\bar{W}\times _{W'}Y'$, then $f_{\bar{Y}}$ is flat over $\pi^{-1}(W'\setminus D')$, in particular,
			$D_2$ is contained in $f_{\bar{Y}}^*\pi_{\bar{Y}}^*D'$. 
			By assumption, $D_3$ is also contained in $f_{\bar{Y}}^*\pi_{\bar{Y}}^*D'$.
			
			Suppose $P$ is an irreducible component of $f_{\bar{Y}}^*\pi_{\bar{Y}}^*D'$. Applying Hurwitz's formula on 
			$\pi_Y^*(K_{Y'}+R' +\mathrm{red}(f_{Y'}^*D'))$
			, locally near $P$, we have
			$$\pi_Y^*(K_{Y'}+R' +\mathrm{red}(f_{Y'}^*D'))=K_{\bar{Y}}+P.$$
			
			Recall that $(Y,\Delta_{Y,> 0})$ is lc. By applying Hurwitz's formula on $\tau_Y$, locally near $P$, we have
			$h_* \tau_Y^* (K_Y+\Delta_{Y,> 0})\leq K_{\bar{Y}}+P$.
			Because $P$ is arbitrary, by comparing the coefficients of each component of $f_{\bar{Y}}^*\pi_{\bar{Y}}^*D'$ on both sides of Equation \eqref{very hard equation}, we have
			\begin{equation}\label{D_2+D_3}
				h_* \tau_Y^* (K_Y+\Delta_Y) +D_2+D_3\leq \pi_Y^*(K_{Y'}+R' +\mathrm{red}(f_{Y'}^*D')).
			\end{equation}
			
			\vspace{0.3cm}
			
			\textit{Step 5}. In this step we show that $D_4=0$.
			
			Suppose $P$ is an irreducible component of $D_4$ and it dominates a prime divisor $\bar{Q}$ on $\bar{W}$, by assumption, $\bar{Q}$ is not contained in $\pi^*D'$. 
			Recall that $\pi$ is a finite morphism, this means $\bar{Q}$ dominates a prime divisor $Q'$ on $W'$ and $Q'$ is not contained in $D'$. Since $\tau$ is finite on $W'\setminus D'$ and $\tau^{-1}(W\setminus D) =W'\setminus D'$, then $Q'$ dominates a divisor $Q$ on $W$ and $Q$ is not contained in $D$.
			
			By assumption, every codimension 1 point in $W\setminus D$ is dominated codimension 1 point in $Y\setminus\Supp(\Delta_{Y,<0})$, then $Q$ is dominated by a divisor $Q_Y$ on $Y$ and $Q_Y$ is not contained in $\Supp(\Delta_{Y,<0})$. It is easy to see that locally over the generic point of $Q$, because $K_Y+\Delta_Y+f^*Q\sim_{\mathbb{Q}} f_Y^*(K_W+Q+\M_W)$ and $Q_Y$ is an lc center of $(Y,\Delta_Y+f^*Q)$ dominating $Q$.
			
			Next we show that $Q_Y$ is the only prime divisor on $Y$ dominating $Q$ and $\mathrm{coeff}_P(h_* \tau_Y^*Q_Y)>0$.

			Because $f_{\bar{W}}$ is smooth over the generic point of $\bar{Q}$, then $\mathrm{coeff}_Pf^*_{\bar{W}}\bar{Q}=1$. Locally over the generic point of $\bar{Q}$, we have  $$K_{\bar{Y}}+\bar{R}+P:=f_{\bar{Y}}^*(K_{\bar{W}}+\bar{Q}+\bar{\M}_{\bar{W}}),$$
			where $\bar{R}:=\pi_{Y}^*R'$ and $\bar{\M}$ is the moduli $\mathbf{b}$-divisor of $f_{\bar{Y}}$. Since $\pi_Y$ is finite, $\bar{R}$ is horizontal over $\bar{W}$. By Proposition \ref{moduli part stable under base change}, because $\M$ descends on $W$ and $\M'$ descends on $W'$, $\bar{\M}_{\bar{W}}=\pi^*\M'_{W'}=\pi^*\tau^*\M_W$.
			
			Recall that by assumption $(Y',R')\rightarrow W'$ is log smooth and sub-klt over $W'\setminus D'$, then it is log smooth and sub-klt over the generic point of $Q'$ and $(\bar{Y},\bar{R})$ is log smooth and sub-klt over the generic point of $\bar{Q}$. Thus, $P$ is the only prime divisor on $\bar{Y}$ and the only lc center of $(\bar{Y},\bar{R}+P)$ which dominates $\bar{Q}$.
			
			Let $\tilde{R}$ be the strict transform of $\bar{R}$ on $\tilde{Y}$, then locally over the generic point of $\bar{Q}$, we have
			$$K_{\tilde{Y}}+\tilde{R}+\tilde{P}\sim_{\mathbb{Q}} h^*(K_{\bar{Y}}+\bar{R}+P)\sim_{\mathbb{Q}}f_{\tilde{Y}}^*(K_{\bar{W}}+\bar{Q}+\bar{\M}_{\bar{W}}),$$
			where $\tilde{P}$ is a $\mathbb{Q}$-divisor such that every irreducible component of $\tilde{P}$ dominates $\bar{Q}$. Because $P$ is the only prime divisor on $\bar{Y}$ and the only lc center of $(\bar{Y},\bar{R}+P)$ which dominates $\bar{Q}$, then $h^{-1}_*P$ is the only component of $\tilde{P}$ with coefficient 1, which is also the only lc center of $(\tilde{Y},\tilde{R}+\tilde{P})$ dominating $\bar{Q}$
			
			Recall that by assumption, $\tau\circ \pi$ is finite over the generic point of $Q$, by the Hurwit'z formula, locally over $Q$, we have $K_{\bar{W}}+\bar{Q}=\pi^*\tau^* (K_W+Q)$. Also because $\M$ descends on $W$ and $\bar{\M}$ descends on $\bar{W}$, we have
			$$K_{\bar{W}}+\bar{Q}+\bar{\M}_{\bar{W}}=\pi^*\tau^* (K_W+Q+\M_W).$$
			Locally over the generic point of $Q$, because $K_Y+\Delta_Y+f^*Q\sim_{\mathbb{Q}} f_Y^*(K_W+Q+\M_W)$, then we have that
			$$K_{\tilde{Y}}+\tilde{R}+\tilde{P}=\tau_Y^*(K_Y+\Delta_Y+f^*Q).$$
			Since $h^{-1}_*P$ is the only lc center of $(\tilde{Y},\tilde{R}+\tilde{P})$ dominating $Q$ and $Q_Y$ is an lc center of $(Y,\Delta_Y+f^*Q)$ dominating $Q$, then by Hurwitz's formula, $h^{-1}_*P$ dominates $Q_Y$. Also because $P$ is the only prime divisor on $\bar{Y}$ dominating $\bar{Q}$, then we have $Q_Y$ is the only divisor on $Y$ dominating $Q$ such that $\mathrm{coeff}_P(h_* \tau_Y^*Q_Y)>0$. Since $Q_Y$ is not contained in $\Delta_{Y,<0}$, then $D_4=0$.
			
			\vspace{0.3cm}
			
			\textit{Step 6}. In this step we complete the proof.
			
			Combining Equation \eqref{D_1}, \eqref{D_2+D_3} and $D_4=0$, we have 
			\begin{equation}\label{all D_i}
				\begin{aligned}
					& h_*\tau^*_Y(K_Y+\Delta_{Y,> 0})\\
					= &h_*\tau^*_Y(K_Y+\Delta_{Y})+D_1+D_2+D_3+D_4 \\
					\leq & \pi_Y^*(K_{Y'}+R'_{> 0}+\mathrm{red}(f_{Y'}^*D')).
				\end{aligned}
			\end{equation}
		\end{proof}
	\end{thm}

	\begin{proof}[Proof of Theorem \ref{bounded base implies bounded fibration}]
	
	\textit{Step 1}. In this step we introduce some notation from the previous theorems in this section.
	
	Fix a positive integer $d$, $\Ci\subset \mathbb{Q}\cap [0,1]$ a DCC set and $v,V$ two positive rational numbers. By Theorem \ref{essential diagram} and Theorem \ref{base is bounded}, we have the following commutative diagram
	$$\xymatrix{
		(\Cx_{\Cs^o},\Cd_{\Cs^o})  \ar[d] _{\Cf}& (\Cx_{\bar{\Cs^o}},\Cd_{\bar{\Cs^o}})\ar[l]_{\tau_\Cx} \ar[r]^{\rho_{\Cx}} \ar[d] &   (\Cx_{{\Ct^o}},\Cd_{{\Ct^o}}) \ar[d]_{\Cf_{{\Ct^o}}} &\\
		\Cs^o 	\ar@/_1pc/[rrr] _\Phi &  \bar{\Cs^o}\ar[l]_{\tau} \ar[r]^{\rho}& 	 {\Ct^o}   \ar[r]^{\pi} &\Cv^o,
	}$$
	where 
	\begin{itemize}
		\item $(\Cx_{\Cs^o},\Cd_{\Cs^o})\rightarrow {\Cs^o},(\Cx_{\Ct^o} ,\Cd_{\Ct^o})\rightarrow {\Ct^o}$ and $(\Cx_{\bar{\Cs^o}},\Cd_{\bar{\Cs^o}})\rightarrow \bar{\Cs}^o$ are locally stable families of log Calabi--Yau pairs over quasi-projective varieties,
		\item $(\Cx_{\Cs^o},\Cd_{\Cs^o})\rightarrow {\Cs^o}$ and $(\Cx_{\Ct^o} ,\Cd_{\Ct^o})\rightarrow {\Ct^o}$ have fiberwise log resolutions,
		\item $\Phi,\rho,\tau$ are surjective, $\pi$ is finite,
		\item $(\Cx_{\bar{{\Cs^o}}},\Cd_{\bar{{\Cs^o}}})\cong (\Cx_{\Cs^o},\Cd_{\Cs^o})\times_{\Cs^o} \bar{{\Cs^o}}\cong (\Cx_{\Ct^o},\Cd_{\Ct^o})\times_{\Ct^o} \bar{{\Cs^o}}$, and
		\item there is a family of relative ample over ${\Cs^o}$ divisors $\{\Gamma_{{\Cs^o},\alpha},\alpha\in \mathscr{C}\}$ on $\Cx_{\Cs^o}$, where $\mathscr{C}$ is a variety with positive dimension, such that $\tau_{\Cx}^* \Gamma_{{\Cs^o},\alpha} =\rho_\Cx^* \Gamma_{{\Ct^o},\alpha}$ for effective relative ample over ${\Ct^o}$ divisors $\{\Gamma_{{\Ct^o},\alpha},\alpha\in \mathscr{C}\}$ on $\Cx_{\Ct^o}$.
	\end{itemize}
	And by Theorem \ref{parametrize moduli map and get universal ambro models}, we have a projective log smooth morphism $(\Cw,\Cd)\rightarrow T$ with reduced boundary over a scheme $T$ of finite type and a morphism $\Theta:\Cw\setminus \Cd\rightarrow \Cv^o$, such that:
	If $(X,\Delta) \in \Cp(d,\Ci,v,V)$ be a log pair together with a fibration $f:X\rightarrow Z$ and $K_X+\Delta\sim_{\mathbb{Q}} f^*(K_Z+B_Z+\M_Z)$, then there is an open subset $U\subset Z$, a morphism $\phi:U\rightarrow \Cs^o$, and a closed point $t\in T$ such that
	\begin{itemize}
		\item $(X_U,\Delta_U)\cong (\Cx_{\Cs^o},\Cd_{\Cs^o})\times _{\Cs^o}U$,
		\item $Z\setminus \Supp(B_Z)$ has an open subset which is isomorphic to a big open subset of $\Cw_t\setminus \Cd_t$,
		\item $\Phi\circ \phi$ is equal to $\Theta|_{\Cw_t\setminus \Cd_t}$ on an open subset of $U\subset Z$, and
		\item $\M$ descends on $\Cw_t$.
	\end{itemize}

	Let $\Cy_{\Cs^o}\rightarrow \Cx_{\Cs^o}$ be a fiberwise log resolution of $(\Cx_{\Cs^o},\Cd_{\Cs^o})$ over $\Cs^o$ and $(\Cy_{\Cs^o},\Cr_{\Cs^o})\rightarrow (\Cx_{\Cs^o},\Cd_{\Cs^o})$ a crepant birational morphism, then $(\Cy_{\Cs^o},\Cr_{\Cs^o})$ is log smooth over $\Cs^o$. Similarily we define $(\Cy_{\Ct^o},\Cr_{\Ct^o})\rightarrow \Ct^o$ to be a fiberwise log resolution of $(\Cx_{\Ct^o},\Cd_{\Ct^o})\rightarrow \Ct^o$. Let $\Cc_{\Cs^o,\alpha}$ be the pullbacks of $\Gamma_{{\Ct^o},\alpha}$ on $\Cy_{{\Cs^o}}$, and $\Cc_{\Ct^o,\alpha}$ be the pullbacks of $\Gamma_{{\Ct^o},\alpha}$ on $\Cy_{{\Ct^o}}$. We pick a positive rational number $\delta$ such that every fiber of $(\Cy_{\Cs^o},\Cr_{\Cs^o}+\delta\Cc_{\Cs^o,\alpha})\rightarrow \Cs^o$ and $(\Cy_{\Ct^o},\Cr_{\Ct^o}+\delta\Cc_{\Ct^o,\alpha})\rightarrow \Ct^o$ is klt, then we replace $\Cc_{{\Cs^o},\alpha}$ by $\delta \Cc_{{\Cs^o},\alpha}$ and $\Cc_{{\Ct^o},\alpha}$ by $\delta \Cc_{{\Ct^o},\alpha}$. Note that $\delta$ only depends on $d,c,v$.
	
	\vspace{0.3cm}
	
	\textit{Step 2}. In this step we construction a projective birational morphism $g:Y\rightarrow X$, a fibration $f_Y:Y\rightarrow \Cw_t$ and a log pair $(Y,\Delta_Y)$ such that $K_Y+\Delta_Y\geq g^*(K_X+\Delta)$, the generic fiber of $f_Y$ is crepant birationally equivalent to the generic fiber of $f$ and $(Y,\Delta)\rightarrow \Cw_t$ satisfies the conditions in Theorem \ref{compare volumes}.
	
	Because $X$ is birationally equivalent to the pullback of $\Cy_{\Cs^o}\rightarrow \Cs^o$ by the morphism $\phi:U\rightarrow \Cs^o$, and $U$ is birationally equivalent to $\Cw_t$, we define
	$f_Y:Y\rightarrow \Cw_t$ to be an extension of the pullback of $\Cy_{\Cs^o}\rightarrow \Cs^o$ by $\phi:U\rightarrow \Cs^o$, and $R$ to be the closure of the pullback of $\Cr_{\Cs^o}$ by $\phi:U\rightarrow \Cs^o$. 
	Since a general fiber $Y_g$ has a morphism to a general fiber $X_g$, we may assume that there is a projective birational morphism $g:Y\rightarrow X$. Let $\Delta_Y$ be the $\mathbb{Q}$-divisor such that $K_Y+\Delta_Y\sim_{\mathbb{Q}}f_Y^*(K_{\Cw_t}+\M_{\Cw_t}+\Cd_t)$. 
	Because $(\Cy_{\Cs^o},\Cr_{\Cs^o,>0})$ is log smooth over $\Cs^o$, we may assume that $(Y,\Delta_{Y})$ is log smooth sub-lc. 
	
	Let $C_\alpha$ denote the closure of the pullback of $\Cc_{\Cs^o,\alpha}$ by $\phi$ in $Y$. Because $\bigcap_{\alpha \in \mathscr{C}} \Cc_{\Cs^o,\alpha}=\emptyset$, by blowing up some subvarieties in $Y$ which do not dominate over $\Cw_t$, we may assume that $(Y,\Delta_{Y,> 0}+C_\alpha+f_Y^* \mathcal{H}_t)$ is lc for every $\alpha \in \mathscr{C}$ and $\bigcap_{\alpha \in \mathscr{C}}\mathrm{Supp}(C_\alpha)=\emptyset$.
	
	Since $Z\setminus \Supp(B_Z)$ has an open subset which is isomorphic to a big open subset of $\Cw_t\setminus \Cd_t$, then every codimension 1 point $Q$ in $\Cw_t\setminus\Cd_t$ maps to a codimension 1 point in $Z\setminus \Supp(B_Z)$, and it is dominated by a codimension 1 point $P$ on $X$. Because the closure of $Q$ is not contained $\Supp(B_Z)$ and by Theorem \ref{inverse of canonical bundle formula} every irreducible component of $B_Z$ is dominated by an irreducible component of $\Delta_{Y,v}$, then the closure of $P$ is not contained in $\Supp(\Delta)$, thus, the image of $P$ on $Y$ is contained in $Y\setminus \Supp(\Delta_Y)$.
	
	\vspace{0.3cm}
	
	\textit{Step 3}. In this step we construct a projective log smooth morphism $(\Cy_{\Cw'},\Cr_{\Cw'})\rightarrow \Cw'$ over $T$ depending only on $d,\Ci,v,V$ such that the generic fiber of $(Y,\Delta_Y)\rightarrow \Cw_t$ is crepant birationally equivalent to the generic fiber of $(\Cy_{\Cw'_t},\Cr_{\Cw'_t})\rightarrow \Cw'_t$ after a finite base change.
	
	Let $\Cu:=\Cw\setminus \Cd$, $\Cu':=\Cu\times_{\Cv^o}\Ct^o$, where the morphism $\Cu\rightarrow \Cv^o$ is $\Theta$, denote the morphism $\Cu'\rightarrow \Ct^o$ by $\Theta'$, let $\Cy_{\Cu'}$ be the pullback of $\Cy_{\Ct^o}\rightarrow \Ct^o$ by $\Theta'$ and $\Cy_{\Cu'_t}$ the fiber over $\Cu'_t$.

	After passing to a stratification of $T$, we choose $\Cw'$ to be a fiberwise smooth compactification of $\Cu'$ over $T$ such that there is a generically finite morphism $\Pi:\Cw'\rightarrow \Cw$ and $\Cd':=\Cw'\setminus \Cu'$ is a relatively snc divisor over $T$.

	Let $(\Cy_{\Cu'},\Cr_{\Cu'}+\Cc_{\Cu',\alpha})\rightarrow \Cu',\alpha\in \mathscr{C}$ be the pullback of $(\Cy_{\Ct^o},\Cr_{\Ct^o}+\Cc_{\Ct^o,\alpha})\rightarrow \Ct^o$ by $\Theta':\Cu'\rightarrow \Ct^o$. Because $(\Cy_{\Ct^o},\Cr_{\Ct^o}+\Cc_{\Ct^o,\alpha})$ is log smooth over $\Ct^o$ for all $\alpha\in \mathscr{C}$, $(\Cy_{\Cu'},\Cr_{\Cu'}+\Cc_{\Cu',\alpha})$ is log smooth over $\Cu'$. Also because $(\Cw',\Cd')\rightarrow T$ is a log smooth morphism, we can choose an extension $(\Cy_{\Cw'},\Cr_{\Cw'}+\Cc_{\Cw',\alpha})$ of $(\Cy_{\Cu'},\Cr_{\Cu'}+\Cc_{\Cu',\alpha})$ and replace $\mathscr{C}$ by an open subset, such that
	\begin{itemize}
		\item $\Cy_{\Cu'}\rightarrow \Cu'$ extends to a morphism $\mathfrak{F}:\Cy_{\Cw'}\rightarrow \Cw'$,
		\item $\Cr_{\Cw'_s}$ is horizontal over $\Cw'_s$, and
		\item $(\Cy_{\Cw'_s},\Cr_{\Cw'_s}+\Cc_{\Cw'_s,\alpha}+\mathrm{red}(f_{\Cw'_s}^*\Cd'_s))$ is log smooth sub-lc, 
	\end{itemize}
	for every closed point $s\in T$ and every $\alpha\in \mathscr{C}$.
	
	Let $\mathcal{L}$ be relatively very ample line bundle on $\Cw$ over $T$ such that $\omega_{{\Cw}_s}\otimes{\mathcal{L}}_s$ is big for every $s\in T$.
	Because $\mathcal{L}_s$ is very ample, we can choose a general member $\mathcal{H}_s\in |\mathcal{L}_s|$ such that $(\Cy_{\Cw'_s},\Cr_{\Cw'_s}+\Cc_{\Cw'_s,\alpha}+f_{\Cw'_s}^*\mathcal{H}'_s+\mathrm{red}(f_{\Cw'_s}^*\Cd'_s))$ is log smooth sub-lc for every closed point $s\in T$, where $\mathcal{H}_s':=(\Pi|_{\Cw'_s})^*\mathcal{H}_s$. By the invariance of plurigenera \cite[Theorem 1.8]{HMX13}, there is a number $v' \gg 0$ such that
	$$\vol(K_{\Cy_{\Cw'_s}}+\Cr_{\Cw'_s,> 0}+\Cc_{\Cw'_s,\alpha}+f_{\Cw'_s}^*\mathcal{H}'_s+\mathrm{red}(f_{\Cw'_s}^*\Cd'_s))\leq v',$$
	for all $s\in T$ and $\alpha \in \mathscr{C}$. Note that $v'$ depends only on $d,\Ci,v,V$
	
	\vspace{0.3cm}
	
	\textit{Step 4}. In this step we construct a log general type log pair $(Y,\Delta_{Y,> 0}+C_0+f_Y^*\mathcal{H}_t)$ and show that $\vol(K_Y+\Delta_{Y,> 0}+C_0+f_Y^*\mathcal{H}_t)\leq v'$ for a closed point $0\in \mathscr{C}$.
	
	Let $\M'$ be the moduli $\mathbf{b}$-divisor corresponding to $(\Cy_{\Cw'_t},\Cr_{\Cw'_t})\rightarrow \Cw'_t$.
	Recall that $(\Cx_{\Cs^o},\Cd_{\Cs^o})\times_{\Cs^o} \bar{{\Cs^o}}\cong (\Cx_{\Ct^o},\Cd_{\Ct^o})\times_{\Ct^o} \bar{{\Cs^o}}$, $\Ct^o\rightarrow \Cv^o$ is a finite morphism and $\phi_{\Cv}$ maps $\Cu_t$ into $\Cv^o$. Because the generic fiber of $(Y,\Delta_Y)\rightarrow \Cw_t$ is equal to the generic fiber of the pullback of $(\Cy_{{\Cs^o}},\Cr_{\Cs^o})\rightarrow \Cs^o$ via $\phi:U\rightarrow \Cs^o$ and the generic fiber of $(\Cy_{\Cw'_t},\Cr_{\Cw'_t})\rightarrow \Cw'_t$ is equal to the generic fiber of the pullback of $(\Cy_{{\Ct^o}},\Cr_{\Ct^o})\rightarrow \Ct^o$ via $\Theta'|_{\Cu'_t}:\Cu'_t\rightarrow \Ct^o$,
	then the generic fiber of $(Y,\Delta_Y)\rightarrow \Cw_t$ is crepant birationally equivalent to the generic fiber of $(\Cy_{\Cw'_t},\Cr_{\Cw'_t})\rightarrow \Cw'_t$ after a finite base change. Denote the moduli $\mathbf{b}$-divisor of $(\Cy_{\Cw'_t},\Cr_{\Cw'_t})\rightarrow \Cw'_t$ by $\M'$, since $\M$ descends on $\Cw_t$, by Proposition \ref{moduli part stable under base change}, $\M'$ descends on $\Cw'_t$.
	Then we can construct the following diagram.
	$$\xymatrix@R=0.3em@C=1.8em{
		\tilde{Y} \ar[rr]^{h_{Y}} \ar[dr] \ar[dddd]_{\tau_{Y}}&& \bar{Y} \ar[dl] \ar[dd]^{\pi_{Y}}\\
		&\bar{W} \ar[dd]&\\
		& & \Cy_{\Cw'_t}\ar[dl]^{f_{\Cw'_s}}\\
		&\Cw'_t \ar[dd]&\\
		Y \ar[dr]_{f_Y}&& \\
		& \Cw_t, &
	}$$
	\begin{itemize}
		\item every variety is projective and normal,
		\item $\bar{W}\rightarrow \Cw'_t$ is finite,
		\item $\M$ descends on $\Cw_t$ and $\M'$ descends on $\Cw_t'$,
		\item $\bar{Y}$ is the the normalization of the main component of $\bar{W}\times _{\Cw'_t}\Cy_{\Cw'_t}$,
		\item after base changing to $\bar{W}$, the generic fiber of $(Y',R')\rightarrow W'$ is crepant birationally equivalent to the generic fiber of $(Y,R)\rightarrow W$, and
		\item $\tilde{Y}$ is a common resolution of $\bar{Y}$ and $Y\times _{\Cw_t} \bar{W}$.
	\end{itemize}
	
	By Theorem \ref{compare volumes}, we have 
	\begin{equation}\label{equation 1 in proof of the main theorem}
		h_{Y*}\tau_Y^*(K_Y+\Delta_{Y,>0}+f_Y^*\mathcal{H}_s)\leq \pi_Y^*(K_{\Cy_{\Cw'_s}}+\Cr_{\Cw'_s,> 0}+f_{\Cw'_s}^*\mathcal{H}'_s+\mathrm{red}(f_{\Cw'_s}^*\Cd'_s)).
	\end{equation}

	Because $\tau_{Y}$ is generically finite, there are only finitely many prime divisor on $\tilde{Y}$ which maps to a codimension $\geq 2$ subvariety on $Y$.
	Note $\bigcap_{\alpha \in \mathscr{C}}\Supp(C_\alpha)=\emptyset$. Also because $\mathscr{C}$ is a variety of positive dimension, then there exists a closed point $0\in \mathscr{C}$ such that $\Supp(C_0)$ does not contain the image of any prime divisor on $\tilde{Y}$, then every irreducible component of $\tau_{Y}^* C_0$ dominates a divisor on $Y$.
	
	Because $\Cc_{\Cs^o}\times_{\Cs^o}\bar{\Cs^o}=\Cc_{\Ct^o}\times_{\Ct^o}\bar{\Cs^o}$, then $\tau_{Y}^* C_0|_{\tilde{Y}_\eta}=h_Y^*\pi_{Y}^* \Cc_{\Cw'_t,0}|_{\bar{Y}_\eta}$, where $\bar{Y}_\eta$ is the generic fiber of $\bar{Y}\rightarrow \bar{W}$ and $\tilde{Y}_\eta$ is the generic fiber of $\tilde{Y}\rightarrow \bar{W}$. 
	Recall that $C_0$ is horizontal over $\Cw_t$ and $\Supp(C_0)$ does not contain the image of any prime divisor on $\tilde{Y}$, then $\tau_{Y}^* C_0$ is horizontal over $\bar{W}$. Also because $h_Y^*\pi_{Y}^* \Cc_{\Cw'_t,0}$ is effective, we have $\tau_Y^* C_0 \leq h_Y^*\pi_{Y}^* \Cc'_{\Cw'_t,0}$, then
	\begin{equation}\label{equation 2 in proof of the main theorem}
		h_{Y*}\tau_Y^* C_0 \leq \pi_{Y}^* \Cc'_{\Cw'_t,0}.
	\end{equation}
	
	Recall that $\mathcal{H}_s':=\Pi_s^*\mathcal{H}_s$, then
	\begin{equation}\label{equation 3 in proof of the main theorem}
		h_{Y*}\tau^*_{Y}f_Y^*\mathcal{H}_t=\pi_Y^*f_{\Cw'_t}^*\mathcal{H}'_t.
	\end{equation} 

	Now combine Equations \eqref{equation 1 in proof of the main theorem}, \eqref{equation 2 in proof of the main theorem}, and \eqref{equation 3 in proof of the main theorem}, we have 
	$$h_{Y*}\tau_Y^*(K_Y+\Delta_{Y,> 0}+C_0+f_Y^*\mathcal{H}_t)\leq \pi_Y^*(K_{\Cy'_{\Cw'_t}}+\Cr'_{\Cw'_t,> 0}+\mathrm{red}(f_{\Cw'_t}^*\Cd'_t)+\Cc'_{\Cw'_t,0}+f_{\Cw'_t}^*\mathcal{H}'_t).$$
	Because $\mathrm{deg}(\tau_Y)\geq \mathrm{deg}(\pi_Y)$ and $\vol(K_{\Cy_{\Cw'_s}}+\Cr_{\Cw'_s,> 0}+\Cc_{\Cw'_s,0}+f_{\Cw'_s}^*\mathcal{H}'_s+\mathrm{red}(f_{\Cw'_s}^*\Cd'_s))\leq v'$, we have that $\vol(K_Y+\Delta_{Y,> 0}+C_0+f_Y^*\mathcal{H}_t)\leq v'$.

	Recall that $K_Y+\Delta\sim_{\mathbb{Q}} f_Y^*(K_{\Cw_t}+\Cd_t+\M_{\Cw_t})$. Because $K_{\Cw_t}+\mathcal{H}_t$ is big and $C_0$ is effective and big over $\Cw_t$, then $K_Y+\Delta_{Y,> 0}+C_0+f_Y^*\mathcal{H}_t$ is big.
	\vspace{0.3cm}
	
	\textit{Step 5}. In this step we finish the proof.

	Because the crepant birational equivalence of a general fiber of $(Y,\Delta_Y)\rightarrow Z$ depends only on $d,c,v$, by \cite[Remark 2.10]{Flo14}, there exists a natural number $l$ depending only on $d,c,v$ such that $l(K_Y+\Delta_Y)\sim l(K_{\Cw_t}+\Cd_t+\M_t)$ is Cartier, in particular, $\mathrm{coeff}(\Delta_{Y,>0})\subset \frac{1}{l}\mathbb{N}\cap [0,1]$ is in a finite set.
	
	Note that $\mathrm{coeff}(\Delta_{Y,> 0})$ is in a finite set depending only on $d,c,v$ and $\mathrm{coeff}(C_0)=\delta$, where $\delta$ also depends only on $d,c,v$, then $\mathrm{coeff}(\Delta_{Y,> 0}+C_0+f_Y^*\mathcal{H}_t)$ is in a finite set depending only on $d,c,v$. By \cite[Lemma 2.3.4]{HMX13}, \cite[Theorem 3.1]{HMX13}, and \cite[Theorem 1.3]{HMX14}, $(Y,\Delta_{Y,> 0}+C_0+f_Y^*\mathcal{H}_t)$ is log birationally bounded, then $(Y,\Delta_{Y,>0})$ is log birationally bounded. In particular, $X$ is birationally bounded.
	
	Recall that by  Theorem \ref{base is bounded} and Theorem \ref{parametrize moduli map and get universal ambro models}, $Z\setminus \Supp(B_Z)$ has an open subset which is isomorphic to a big open subset of $\Cw_t\setminus \Cd_t$, then $\Cd_t$ contains the exceptional divisor of $\Cw_t\dashrightarrow Z$. Let $W\rightarrow Z$, $W\rightarrow \Cw_t$ be a common resolution of $\Cw_t$ and $Z$, then by the negativity lemma, the pullback of $K_{\Cw_t}+\Cd_t+\M_t$ on $W$ is larger than or equal to the pullback of $K_Z+B_Z+\M_Z$ on $W$. Since $K_Y+\Delta_Y\sim_{\mathbb{Q}} f^*(K_{\Cw_t}+\Cd_t+\M_t)$ and there is a birational contraction $g:Y\rightarrow X$, we have $K_Y+\Delta_Y\geq h^*(K_X+\Delta)$ and $\mathbf{L}_{\Delta,Y}\leq \Delta_{Y,>0}$.
	
	Because $(Y,\Delta_{Y,>0})$ is log birationally bounded, there exists a log smooth morphism $(\Cy,\Cq)\rightarrow S$ with reduced boundary depending only on $d,\Ci,V,v$, where $S$ is of finite type, and a closed point $s\in S$ such that there is a birational map $\Cy_s\dashrightarrow Y$ and $\Cq_s$ contains the strict transform of $\Delta_{Y,>0}$ and the exceptional divisor of $\Cy_s\dashrightarrow Y$. Because $(Y,\Delta_{Y,>0})$ is lc, then $\Cq_s\geq \mathbf{L}_{\Delta_{Y,>0},\Cy_s}$. Also because $\mathbf{L}_{\Delta,Y}\leq \Delta_{Y,>0}$, we have $\Cq_s\geq \mathbf{L}_{\Delta,\Cy_s}$.
	\end{proof}

	\section{Proof of Theorem \ref{main theorem 1}}
	\begin{thm}\label{bounded family and good minimal model}
		Let $d$ be a positive integer, $\Ci\subset [0,1]\cap \mathbb{Q}$ a DCC set, and $v,u$ two positive rational numbers. Let $\Cg_{klt}(d,\Ci,v,u)$ be the set of log pairs defined in Definition \ref{definition of good minimal models}. 
		
		Then there exists a projective log smooth morphism $(\Cy,\Cb)\rightarrow T$, where $\Cb$ is a $\mathbb{Q}$-divisor depending only on $d,\Ci,v,u$, such that for every $(X,\Delta)\in \Cg_{klt}(d,\Ci,v,u)$, there is a closed point $s\in T$ such that $(\Cy_s,\Cb_s)$ has a good minimal model which is crepant birationally equivalent to $(X,\Delta)$.
		
		\begin{proof}
			Fix $(X,\Delta)\in \Cg_{klt}(d,\Ci,v,u)$ with the fibration $f:X\rightarrow Z$ and $K_X+\Delta\sim_{\mathbb{Q}} f^*(K_Z+B_Z+\M_Z)$. Because $K_Z+B_Z+\M_Z$ is ample, by the main result of \cite{Bir21b} and \cite{Jia21}, there is an integer $r>0$ depending only on $d,\Ci,v,u$ such that $r(K_Z+B_Z+\M_Z)$ is very ample. Define $H:=r(K_Z+B_Z+\M_Z)$, then $\vol(H+K_Z+B_Z+\M_Z)\leq (r+1)^n\vol(K_Z+B_Z+\M_Z) \leq (r+1)^n\mathrm{Ivol}(K_X+\Delta)=(r+1)^nv$, thus Theorem \ref{bounded base implies bounded fibration} applies. 
		
			Define $V:=(r+1)^nv$, by Theorem \ref{bounded base implies bounded fibration}, there is a projective log smooth morphism $(\Cy,\Cq)\rightarrow S$ with reduced boundary depending only on $d,\Ci,v,V$ such that there is a closed point $s\in S$ and a birational map $g:\Cy_s \dashrightarrow X$ such that $\mathbf{L}_{\Delta,\Cy_s}\leq \Cq_s$.

			Because the crepant birational equivalence of a general fiber of $(X,\Delta)\rightarrow Z$ is in a log bounded family depending only on $d,c,v$ and $r(K_Z+B_Z+\M_Z)$ is Cartier, by \cite[Remark 2.10]{Flo14}, there exists a natural number $l$ depending only on $d,c,v$ such that $l(K_X+\Delta)\sim l(K_Z+B_Z+\M_Z)$ is Cartier. Also because $(X,\Delta)$ is klt, we have $\mathbf{L}_{\Delta,\Cy_s}\leq (1-\frac{l}{l})\Cq_s$. 
			
			Let $\chi:\Cy'\rightarrow \Cy$ be a projective birational morphism such that
			\begin{enumerate}
				\item $\chi$ only blows up strata of $(\Cy,\Cq)$, and
				\item $(\Cy',\mathbf{L}_{(1-\frac{1}{l})\Cq,\Cy'})$ has terminal singularities.
			\end{enumerate}
			Let $p': X'\rightarrow \Cy'_s$ and $q:X'\rightarrow X$ be a common resolution, $p:=\chi_s\circ p':X'\rightarrow \Cy_s$. 
			$$\xymatrix@R=0.5em@C=2.4em{
				&X'\ar[dd]_p \ar[dl]_{p'} \ar[dr]^q&\\
				\Cy'_s\ar[dr]_{\chi_s}& &  X\\
				&  \Cy_s \ar@{-->}[ur]_{g}.&
			}$$
			Write $K_{X'}+\Delta'\sim_{\mathbb{Q}}q^*(K_X+\Delta)$, then $\mathbf{L}_{\Delta,\Cy_s}=p_*\Delta'_{> 0}$. Because $K_X+\Delta$ is nef, by the negativity lemma, $K_{X'}+\Delta'\leq p^*(K_{\Cy_s}+p_*\Delta')= p^*(K_{\Cy_s}+\mathbf{L}_{\Delta,\Cy_s})$. Also because $\mathbf{L}_{\Delta,\Cy_s} \leq (1-\frac{1}{l})\Cq_s$, then $\mathbf{L}_{\Delta,\Cy'_s}\leq \mathbf{L}_{(1-\frac{1}{l})\Cq_s,\Cy'_s}$ and $(\Cy'_s, \mathbf{L}_{\Delta,\Cy'_s})$ has terminal singularities. We replace $(\Cy,\Cq)$ by $(\Cy',\Supp(\mathbf{L}_{(1-\frac{1}{l})\Cq,\Cy'}))$ and assume that $(\Cy_s,\mathbf{L}_{\Delta,\Cy_s})$ has terminal singularities. 
			
			Next, we prove that $(\Cy_s,\mathbf{L}_{\Delta,\Cy_s})$ has a good minimal model and the good minimal model is crepant birationally equivalent to $(X,\Delta)$. 
			
			Because $K_X+\Delta$ is semiample, then $(X',\mathbf{L}_{\Delta,X'})$ has a good minimal model.
			By the definition of the $\mathbf{b}$-divisor $\mathbf{L}_{\Delta}$, $\mathbf{L}_{\Delta,\Cy_s}=p_*  \mathbf{L}_{\Delta,X'}$. Since $(\Cy_s, \mathbf{L}_{\Delta,\Cy_s})$ has terminal singularities, $K_{X'}+ \mathbf{L}_{\Delta,X'} -p^*(K_{\Cy_s}+ \mathbf{L}_{\Delta,\Cy_s} )$ is effective and $p$-exceptional. Then by \cite[Lemma 2.10]{HMX13}, $(\Cy_s,\mathbf{L}_{\Delta,\Cy_s} )$ has a good minimal model since $(X', \mathbf{L}_{\Delta,X'})$ has a good minimal model, and it is easy to see that a good minimal model of $(\Cy_s,\mathbf{L}_{\Delta,\Cy_s} )$ is crepant birationally equivalent to $(X,\Delta)$.
			
			By the construction, we have $\mathrm{coeff}(\mathbf{L}_{\Delta,\Cy_s})\subset \{0,\frac{1}{l},\cdots,\frac{l-1}{l}\}$ and $\mathbf{L}_{\Delta,\Cy_s}\leq \Cq_s$. Write $\Cq=\sum_{k=1}^m\Cq_k$ as the sum of all irreducible components. For a vector $a=\{a_k\}_{1\leq k\leq m}$, define $a\Cq=\sum_{k=1}^m a_k\Cq_k$. Let $J$ be the set of all vectors $a$ of dimension $m$ such that the coefficient of each element of $a$ is in $\{0,\frac{1}{l},\cdots,\frac{l-1}{l}\}$, then $J$ is a finite set. We define $T:=J\times S$, and $\Cb$ to be the $\mathbb{Q}$-divisor supported on $J\times \Cq$ such that $\Cb|_{ \{a\}\times S}=a\Cq$ for all $a\in J$. Then $(\Cy,\Cb)\rightarrow T$ satisfies the requirement.
		\end{proof}
	\end{thm}
	\begin{proof}[Proof of Theorem \ref{main theorem 1}]
		Let $(\Cy,\Cb)\rightarrow T$ be the projective log smooth morphism defined in Theorem \ref{bounded family and good minimal model}. To prove the boundedness of $\Cg(d,\Ci,v,u)$, we are free to pass to a stratification of $T$ and assume that $T$ is smooth and irreducible.

		Let $(X,\Delta)\in \Cg(d,\Ci,v,u)$, suppose $s\in T$ is the corresponding closed point, by Theorem \ref{bounded family and good minimal model}, a good minimal model of $(\Cy_s,\Cb_s)$ is crepant birationally equivalent to $(X,\Delta)$.
		Because $(\Cy,\Cb)\rightarrow T$ is a log smooth morphism, by \cite[Theorem 1.2]{HMX18}, $(\Cy,\Cb)$ has a good minimal model $(\Cy^m,\Cb^m)$ over $T$. Because $T$ is smooth, by \cite[Theorem 3.1]{HMX18}, $(\Cy^m_u,\Cb^m_u)$ is a semi-ample model of $(\Cy_u,\Cb_u)$ for every closed point $u\in T$ and $(X,\Delta)$ is crepant birationally equivalent to $(\Cy^m_s,\Cb^m_s)$. Therefore, $(X,\Delta)$ is bounded modulo crepant birational equivalence.
	\end{proof}

	\section{Rationally connected Calabi--Yau pairs}
	Recall that a Calabi--Yau variety is a smooth, projective, simply connected variety $Y$ such that $K_Y\sim 0$ and $H^i(Y,\mathcal{O}_Y)=0$ for $0<i<\mathrm{dim}Y$.
	\begin{cor}{\cite[Corollary 5.1]{BDCS20}}\label{rationally connected}
		Let $Y$ be a Calabi--Yau variety. Assume that $Y$ is endowed with a morphism $f: Y\rightarrow Z$ of relative dimension $0< d<\mathrm{dim} Y$. Then $Z$ is rationally connected.
	\end{cor}
	By Corollary \ref{rationally connected}, it is easy to see that Theorem \ref{boundedness of Calabi--Yau with polarized fibration structure} is a special case of the following theorem.
	\begin{thm}\label{boundedness of Calabi--Yau with polarized fibration structure:stronger version}
		Fix positive integers $d,l$ and a positive rational number $v$. Then the set of projective varieties $Y$ such that
		\begin{itemize}
			\item $Y$ is terminal of dimension $d$,
			\item $lK_Y\sim 0$,
			\item $f:Y\rightarrow Z$ is a fibration, 
			\item $Z$ is rationally connected, and
			\item there is an integral divisor $A$ on $Y$ such that $A_g:=A|_{Y_g}$ is ample and $\vol(A_g)=v$, where $Y_g$ is a general fiber of $f$,
		\end{itemize}
		is bounded modulo flops.
	\end{thm}

	\begin{defn}[{\cite[Definition 2.2]{Bir18}}]
		Let $d,r$ be natural numbers and $\epsilon\in [0,1]$ be a positive rational number. A generalized $(d,r,\epsilon)$-Fano type fibration consists of a projective generalized log pair $(Z,\Delta+\M_Z)$ and a fibration $f: Z\rightarrow W$ such that
		\begin{itemize}
			\item $(Z,\Delta+\M_Z)$ is a projective generalized $\epsilon$-lc pair of dimension $d$,
			\item $K_Z+\Delta+\M_Z\sim _{\mathbb{Q}}f^*L$ for some $\mathbb{Q}$-divisor $L$,
			\item $-K_Z$ is big over $W$, i.e. $Z$ is of Fano type over $W$,
			\item $A$ is very ample divisor on $W$ with $A^{\mathrm{dim}W}\leq r$, and
			\item $A-L$ is ample.
		\end{itemize}
	\end{defn}
	
	\begin{thm}[{\cite[Theorem 2.3]{Bir18}}\label{bound Fano type fibration}]
		Let $d,r$ be natural numbers and $\epsilon,\tau\in [0,1]$ be positive rational numbers. Consider the set of all generalized $(d,r,\epsilon)$-Fano type fibrations $(Z,\Delta+\M_Z)\rightarrow W$ such that 
		\begin{itemize}
			\item we have $0\leq B \leq \Delta$ whose non-zero coefficients are $\geq \tau$, and
			\item $-(K_Z+B)$ is big over $W$.
		\end{itemize}
		Then the set of such log pairs $(Z,B)$ is log bounded. 
	\end{thm}
	
	Recall that a birational contraction is a birational map $\phi:Z\dashrightarrow Z'$ such that $\phi^{-1}$ does not contract any divisor.
	
	\begin{thm}[{\cite[Theorem 3.1]{BDCS20}}\label{Fano fibration tower}]
		Let $(Z,\Delta)$ be a projective klt Calabi--Yau pair with $\Delta\neq 0$. Then there exists a birational contraction 
		$$\pi :Z\dashrightarrow Z'$$
		to a $\mathbb{Q}$-factorial log Calabi--Yau pair $(Z',\Delta':=\pi_*\Delta)$, $\Delta'\neq 0$ and a tower of morphisms
		$$Z'=Z_0\xrightarrow{p_0}Z_1\xrightarrow{p_1}Z_2\xrightarrow{p_2}\cdots\xrightarrow{p_{k-1}}Z_k$$
		such that
		\begin{itemize}
			\item for any $1\leq i<k$ there exists a boundary $\Delta_i\neq 0$ on $Z_i$ and $(Z_i,\Delta_i)$ is a klt Calabi--Yau pair,
			\item for any $0\leq i<k$ the morphism $p_i:Z_i\rightarrow Z_{i+1}$ is a Mori fiber space, with $\rho(Z_i/Z_{i+1})=1$, and 
			\item either $\mathrm{dim}Z_{k}=0$, or $\mathrm{dim}Z_k>0$ and $Z_k$ is a klt variety with $K_{Z_{k}}\sim_{\mathbb{Q}}0$.
		\end{itemize}
	\end{thm}
	\begin{prop}[{\cite[Proposition 3.7]{BDCS20}}\label{3.7}]
		Let $(X,\Delta)$ be a projective klt pair and let $f:X\rightarrow Z$ be a fibration between normal varieites. Assume that $K_X+\Delta\sim_{\mathbb{Q}} 0$ and $Z$ is $\mathbb{Q}$-factorial. Let $Z\dashrightarrow Z'$ be a birational contraction of normal projective varieties. Then there exist a projective $\mathbb{Q}$-factorial klt pair $(X',\Delta')$ which is isomorphic to $(X,\Delta)$ in codimension 1 and a projective fibration of normal varieties $f':X'\rightarrow Z'$.
	\end{prop}
	\begin{thm}[{\cite[Theorem 4.1]{BDCS20}}]\label{rationally connected implies bounded}
		Fix positive integers $d,l$. Consider varieties $Z$ such that
		\begin{itemize}
			\item $Z$ is projective klt of dimension $d$,
			\item $Z$ is rationally connected, and
			\item $lK_Z\sim 0$.
		\end{itemize}
		Then the set of such $Z$ is bounded up to flops.
	\end{thm}

	\begin{thm}\label{boundedness of base in induction step}
		Assume Theorem \ref{boundedness of Calabi--Yau with polarized fibration structure:stronger version} in dimension $\leq d-1$. Fix a positive integer $l$ and a positive rational number $v$.  Then there exists a projective morphism $\Cz\rightarrow T$, where $T$ is of finite type, such that if $f: Y\rightarrow Z$ is a fibration between normal projective varieties with the following properties
		\begin{itemize}
			\item $Y$ is klt of dimension $d$,
			\item $lK_Y\sim 0$,
			\item $Z$ is rationally connected, and
			\item there is an integral divisor $A$ on $Y$ such that $A_g:=A|_{Y_g}$ is ample and $\vol(A_g)=v$, where $Y_g$ is a general fiber of $f$.
		\end{itemize}
		Then there is a closed point $t\in T$ and a birational contraction $Z\dashrightarrow \Cz_t$.
		
		\begin{proof}
			Applying the canonical bundle formula on $f$, there is a projective generalized klt pair $(Z,B_Z+\M_Z)$ such that 
			$$K_Z+B_Z+\M_Z\sim_{\mathbb{Q}} 0.$$
			By the main result of \cite{Amb05}, $\M$ is $b$-nef and $b$-abundant. Also because $(Z,B_Z+\M_Z)$ is generalized klt, then there exists a $\mathbb{Q}$-divisor $B\sim_{\mathbb{Q}}B_Z+\M_Z$ such that $(Z,B)$ is a klt pair.

			If $B=0$, by \cite[Corollary 1.6]{Bir23}, a general fiber of $f$ is in a bounded family, then the Cartier index $r$ of $\M$ is bounded. Then by \cite[Remark 2.10]{Flo14}, there is a natural number $l'$ such that $l'K_{Z}\sim 0$. By Theorem \ref{rationally connected implies bounded}, $Z$ is bounded up to flops.
			
			If $B\neq 0$, by Theorem \ref{Fano fibration tower}, there exists a birational contraction
			$$\pi :Z\dashrightarrow Z'$$
			to a $\mathbb{Q}$-factorial log Calabi--Yau pair $(Z',B':=\pi_*B)$, $B'\neq 0$ and a tower of fibrations
			$$Z'=Z_0\xrightarrow{p_0}Z_1\xrightarrow{p_1}Z_2\xrightarrow{p_2}\cdots\xrightarrow{p_{k-1}}Z_k.$$
			To prove the result, we only need to show $Z_0$ is bounded up to flops.
			
			Let $f':Y'\rightarrow Z',g:Y'\rightarrow Y$ be the normalization of the closure of the graph of $\pi\circ f:Y\dashrightarrow Z'$, write $K_{Y'}+E\sim_{\mathbb{Q}}g^*K_Y+F$, where $E$ and $F$ are effective without common components. Because $Z\dashrightarrow Z'$ is a birational contraction, the image of $\mathrm{Exc}(g)$ on $Z'$ does not contain any codimension 1 point of $Z'$. We fix a positive number $\delta\ll 1$ and run $K_{Y'}+(1+\delta)E\sim_{\mathbb{Q}} \delta E+F$-MMP with respect to a relative ample divisor over $Z'$, by \cite[Theorem 1.8]{Bir12}, it terminates with a model $W$ such that $E_W+F_W=0$, and $K_W\sim_{\mathbb{Q}}0$. Also because $lK_Y\sim 0$, then $(W,0)$ is crepant birationally equivalent to $(Y,0)$ and $lK_W\sim 0$. Because a general fiber of $W\rightarrow Z'$ is isomorphic to a general fiber of $f: Z\rightarrow Y$, it is easy to see that $W\rightarrow Z'=Z_0$ satisfies all the conditions in Theorem \ref{boundedness of Calabi--Yau with polarized fibration structure:stronger version}. 
			
			We have a tower of fibrations
			$$W\rightarrow Z_0\xrightarrow{p_0}Z_1\xrightarrow{p_1}Z_2\xrightarrow{p_2}\cdots\xrightarrow{p_{k-1}}Z_k.$$
			Let $e_j$ denote the fibration $W\rightarrow Z_j$. Suppose $g\in Z_j$ is a general closed point, because $Z_0\rightarrow Z_j$ is a composition of Mori fiber space, then a general fiber $Z_{0g}$ of $Z_0\rightarrow Z_j$ is rationally connected. By assumption, suppose $W_g$ is the fiber of $W\rightarrow Z_j$ over $g$, then $W_g$ is klt of dimension $<d$ and $lK_{W_g}\sim 0$. Because $g$ is a general point, a general fiber of $W_g\rightarrow Z_{0g}$ is also a general fiber of $W\rightarrow Z_0$, then by Theorem \ref{boundedness of Calabi--Yau with polarized fibration structure:stronger version} in dimension $\leq d-1$, a general fiber of $e_j: W\rightarrow Z_j$ is bounded modulo flops for all $j\leq k-1$. Thus, the Cartier index of the moduli part of $e_j$ is bounded. 
			
			Because $W$ is terminal, $lK_W\sim 0$ and the Cartier index of the moduli part of $e_j$ is bounded, there exist a natural number $l_j$ and a generalized klt pair $(Z_j,B_j+\M_{j,Z_j})$ on $Z_j$ such that $l_j(K_{Z_j}+B_j+\M_{j,Z_j})\sim 0$ and $l_j\M_j$ is $\mathbf{b}$-Cartier. Because $l_j(K_{Z_j}+B_j+\M_{j,Z_j})$ is Cartier, define 
			$$\epsilon=\min \{\frac{1}{l_j},j=0,...,k\},$$ 
			then $\epsilon$ depends only on $l,d$, and $(Z_j, B_j+\M_{j, Z_j})$ is $\epsilon$-lc for each $j=0,\cdots,k$.
			
			If $\mathrm{dim}{Z_k}=0$, then clearly ${Z_k}$ is bounded set.
			If $\mathrm{dim}{Z_k}>0$, by Theorem \ref{Fano fibration tower}, we have $K_{Z_k}\sim_{\mathbb{Q}} 0$. Because $l_k$ is the integer such that $l_k(K_{Z_k}+B_{Z_k}+\M_{Z_k})\sim 0$ and $B_{Z_k}+\M_{Z_k}$ is effective, then $l_kK_{Z_k}\sim 0$. By Theorem \ref{rationally connected implies bounded}, ${Z_k}$ is bounded up to flops.
			
			Next we show that if $Z_j$ is bounded up to flops, then $Z_{j-1}$ is also bounded up to flops for all $j=1,\cdots,k$. Thus, because ${Z_k}$ is bounded up to flops, we have $Z_0$ is bounded up to flops.
			
			Suppose $Z_j$ is bounded up to flops, then there exists a klt variety $Z'_j$ which is isomorphic to $Z_j$ in codimension 1 and is in a bounded family. Because $\phi_j: Z'_j\dashrightarrow Z_j$ is an isomorphism in codimension 1 between projective varieties, it is also a birational contraction. Also since $Z_j$ is $\mathbb{Q}$-factorial, we can apply Proposition \ref{3.7} to the Mori fiber space $Z_{j-1}\rightarrow Z_j$ and obtain a commutative diagram
			$$\xymatrix@R=1.2em@C=1.8em{
				Z'_{j-1}\ar@{-->}[r]\ar[d]   &   Z_{j-1} \ar[d]\\
				Z'_j    \ar@{-->}[r]^{\phi_j}   &   Z_j ,    
			}$$
			where the horizontal arrow $Z'_{j-1}\dashrightarrow Z_{j-1}$ is an isomorphism in codimension 1 of $\mathbb{Q}$-factorial projective varieties. As all horizontal arrows in the diagram are isomorphisms in codimension 1 of $\mathbb{Q}$-factorial projective varieties, it follows that $\rho(Z_{j-1}/Z_j)=\rho(Z'_{j-1}/Z'_j)=1$; hence $Z'_{j-1}\rightarrow Z'_j$ is a $K_{Z'_{j-1}}$-Mori fiber space. 
			
			Because $Z'_j$ is in a bounded family, there exists a number $r_j$ and a very ample divisor $A_j$ on $Z'_j$ such that $A_j^{\mathrm{dim}Z'_j}\leq r_j$. Also because $(Z_{j-1},B_{j-1}+\M_{j-1,Z_{j-1}})$ is $\epsilon$-lc, $Z'_j$ is isomorphic to $Z_j$ in codimension 1 and $K_{Z_{j-1}}+B_{j-1}+\M_{j-1,Z_{j-1}}\sim_{\mathbb{Q}}0$, then $(Z_{j-1}',(\phi_{j-1})_* ^{-1}B_{j-1}+\M_{{j-1},Z'_{j-1}})$ is $\epsilon$-lc. Thus, $Z_{j-1}'\rightarrow Z'_j$ is a generalized $(\mathrm{dim}Z'_{j-1},r_j,\epsilon)$-Fano type fibration. By Theorem \ref{bound Fano type fibration}, $Z'_{j-1}$ is bounded, then $Z_{j-1}$ is bounded up to flops.
		\end{proof}
	\end{thm}
	\begin{proof}[Proof of Theorem \ref{boundedness of Calabi--Yau with polarized fibration structure:stronger version}]
		We prove this by induction on dimension. Assume the result holds in dimension $\leq d-1$.
		
		Let $\Cz\rightarrow T$ be the bounded family defined in Theorem \ref{boundedness of base in induction step}. By the definition of boundedness, we may assume that there is a positive number $V>0$ and a relative very ample divisor $\mathcal{H}$ on $\Cz$ over $T$ such that $\vol(\mathcal{H}_s)\leq V$
		for all $s\in T$.
		
		Let $\M$ be the moduli $\mathbf{b}$-divisor of $Y\rightarrow Z$. Since $K_Y\sim _{\mathbb{Q}}0$, then 
		$$K_Z+B_Z+\M_Z\sim_{\mathbb{Q}} 0.$$
		Let $p:W\rightarrow Z,q:W\rightarrow \Cz_t$ be a common resolution of $Z\dashrightarrow \Cz_t$, write
		$$K_W+B_W+\M_W\sim_{\mathbb{Q}}p^*(K_Z+B+\M_Z)\sim_{\mathbb{Q}}0.$$
		Because $B\geq 0$, then $B_{W,< 0}$ is $p$-exceptional. Let $B_{\Cz_t}=q_*B_W$, then 
		$$K_{\Cz_t}+B_{\Cz_t}+\M_{\Cz_t}\sim _{\mathbb{Q}}0.$$
		Because $Z\dashrightarrow \Cz_t$ is a birational contraction, then $B_{W,< 0}$ is also $q$-exceptional and $B_{\Cz_t}\geq 0$.
		
		Let $h:Y'\rightarrow Y$ be the closure of the graph of $Y\dashrightarrow \Cz_t$ and write $K_{Y'}+\Delta'\sim_{\mathbb{Q}}h^*K_Y$. Because $Z\dashrightarrow \Cz_t$ is a birational contraction and $Y\rightarrow Z$ is a fibration, then $\Delta_{Y',<0}$ is 
		very exceptional over $\Cz_t$. By \cite[Theorem 1.8]{Bir12}, we can run $K_{Y'}+\Delta_{Y',> 0}$-MMP over $\Cz_t$, which terminates with a model $Y''$ such that $\Delta_{Y''}\geq 0$, then we have $K_{Y''}+\Delta_{Y''}\sim_{\mathbb{Q}}0$. Also because $lK_Y\sim0$, then $l(K_{Y''}+\Delta_{Y''})\sim 0$. In particular, $\mathrm{coeff}\Delta_{Y''}$ is in a finite set.

		Notice that $K_{\Cz_t}+B_{\Cz_t}+\M_{\Cz_t}$ is nef and
		$$\vol(\mathcal{H}_t+K_{\Cz_t}+B_{\Cz_t}+\M_{\Cz_t}) = \vol(\mathcal{H}_t)\leq  V.$$ 
		Also because a general fiber of $Y''\rightarrow \Cz_t$ is isomorphic to a general fiber of  $Y\rightarrow Z$, then by Theorem \ref{bounded base implies bounded fibration}, $Y''$ is birationally bounded. Let $\Cy\rightarrow S$ be the corresponding bounded family, then $Y$ is birational equivalent to $\Cy_s$ for a closed point $s\in S$.
		
		After taking a resolution of the generic fiber of $\Cy\rightarrow S$ and passing to a stratification of $S$, we may assume that $\Cy\rightarrow S$ is a projective smooth morphism. Because $K_Y\sim_{\mathbb{Q}}0$ and both $Y$ and $(\Cy_s,0)$ have terminal singularties, then $Y$ is a minimal model of $(\Cy_s,0)$. By \cite[Theorem 2.12]{HX13}, $\Cy$ has a good minimal model over $T$, denote it by $\Cy'\rightarrow S$. By \cite[Lemma 2.4]{HX13}, $(Y,0)$ is crepant birationally equivalent with $(\Cy'_s,0)$ and $Y$ is isomorphic in codimension 1 with $\Cy'_s$. Thus, $Y$ is bounded up to flops.
	\end{proof}


\begin{thebibliography}{99}
		
		\bibitem[AK00]{AK00} D. Abramovich and K. Karu, \textit{Weak semistable reduction in characteristic 0}, Invent. Math. \textbf{139} (2000), no. 2, 241--273, MR1738451, Zbl 0958.14006.
		
		
		
		\bibitem[Amb05]{Amb05} F. Ambro, \textit{The moduli b-divisor of an lc-trivial fibration}, Compos. Math. \textbf{141} (2005), no. 2, 385--403, MR2134273, Zbl 1094.14025.
		
		
		\bibitem[Bir12]{Bir12} C. Birkar, \textit{Existence of log canonical flips and a special LMMP}, Pub. Math. IHES., \textbf{115} (2012), 325--368, MR2929730, Zbl 1256.14012.
		
		
		
		\bibitem[Bir19]{Bir19} C. Birkar, \textit{Anti-pluricanonical systems on Fano varieties}. Ann. of Math. (2), \textbf{190} (2019), 345--463, MR3997127, Zbl 1470.14078.
		
		
		
		
		\bibitem[Bir21a]{Bir21a} C. Birkar, \textit{Singularities of linear systems and boundedness of Fano varieties}, Ann. of Math. \textbf{193} (2021), no. 2, 347--405, MR4224714, Zbl 1469.14085.
		
		
		\bibitem[Bir21b]{Bir21b} C. Birkar, \textit{Boundedness and volume of generalized pairs}, arXiv: 2103.14935v2.
		
		\bibitem[Bir22a]{Bir18} C. Birkar, \textit{Boundedness of Fano type fibrations
		}, To appear in Ann. Sci. ENS, arXiv:2209.08797.
		
		\bibitem[Bir22b]{Bir22} C. Birkar, \textit{Moduli of algebraic varieties}, arXiv:2211.11237.
		
		\bibitem[Bir23]{Bir23} C. Birkar, \textit{Geometry of polarised varieties}, Pub. Math. IHES, \textbf{137} (2023), 47–105, MR4588595, Zbl 1531.14048.
		
		
		\bibitem[BDCS20]{BDCS20} C. Birkar, G. Di Cerbo, and R. Svaldi, \textit{Boundedness of elliptic Calabi--Yau varieties with a rational section}, arXiv: 2010.09769v1.
		
		
		\bibitem[BZ16]{BZ16} C. Birkar and D.-Q. Zhang, \textit{Effectivity of Iitaka fibrations and pluricanonical systems of polarized pairs}, Pub. Math. IHES., \textbf{123} (2016), 283--331, MR3502099, Zbl 1348.14038.
		
		\bibitem[CDC$\rm{H}^+$18]{CDCHJ18} Weichung Chen, Gabriele Di Cerbo, Jingjun Han, Chen Jiang, and Roberto Svaldi, \textit{Birational boundedness of rationally connected Calabi--Yau 3 folds}, Adv. Math, \textbf{378} (2021), 107541, MR4191257, Zbl 1477.14066.
		
		
		
		
		
		
		
		
		\bibitem[Fil24]{Fil20} S. Filipazzi, \textit{On the boundedness of $n$-folds with $\kappa(X)=n-1$}, Algebr. Geom. \textbf{11} (2024), no.3, 318--345, MR4742322, Zbl 07845150.
		
		
		
		
		
		
		
		
		
		\bibitem[Flo14]{Flo14} E. Floris, \textit{Inductive approach to effective $b$-semiampleness}, Int. Math. Res. Not. IMRN \textbf{6} (2014), 1465-1492, MR3180598, Zbl 1325.14018.
		
		
		
		
		
		
		
		
		\bibitem[FM00]{FM00} O. Fujino and S. Mori, \textit{A canonical bundle formula}, J. Differential Geom. \textbf{56} (2000), no. 1, 167--188, MR1863025, Zbl 1032.14014.
		
		
		
		\bibitem[Gon11]{Gon11} Y. Gongyo, \textit{On the minimal model theory for dlt pairs of numerical kodaira dimension zero}, Math. Rest. Lett. \textbf{18} (2011), no. 5, 991--1000, MR2875871, Zbl 1246.14026.
		
		
		
		
		\bibitem[HMX13]{HMX13} C. D. Hacon, J. M\textsuperscript{c}Kernan, and C. Xu, \textit{On the birational automorphisms of varieties of general type}, Ann. of Math. (2) \textbf{177} (2013), no. 3, 1077--1111, MR3034294, Zbl 1281.14036.
		
		\bibitem[HMX14]{HMX14} C. D. Hacon, J. M\textsuperscript{c}Kernan, and C. Xu, \textit{ACC for log canonical thresholds}, Ann. of Math. \textbf{180} (2014), no. 2, 523--571, MR3224718, MR3224718.
		
		\bibitem[HMX18]{HMX18} C. D. Hacon, J. M\textsuperscript{c}Kernan, and C. Xu, \textit{Boundedness of moduli of varieties of general type}, J. Eur. Math. \textbf{20} (2018), no. 4, 865–901, MR3779687, Zbl 1464.14038.
		
		
		
		\bibitem[HX13]{HX13} C. D. Hacon and C. Xu, \textit{Existence of log canonical closures}, Invent. Math. \textbf{192} (2013), no. 1, 161--195, MR3032329, Zbl 1282.14027.
		
		
		
		
		
		
		
		
		
		
		
		
		
		
		
		
		
		\bibitem[Jia21]{Jia21} J. Jiao, \textit{On the boundedness of canonical models}, arXiv: 2103.13609v3.
		
		
		
		
		\bibitem[Kaw98]{Kaw98} Y. Kawamata, \textit{Subadjunction of log canonical divisors, II}, Amer. J. Math. \textbf{120} (1998), 893--899, MR1646046, Zbl 0919.14003.
		
		
		
		
		\bibitem[Kol07]{Kol07} J. Koll\'ar, \textit{“Kodaira’s canonical bundle formula and adjunction}. In: \textit{Flips for 3-folds and 4-folds}. Ed. by A. Corti. Vol. 35. Oxford Lecture Series in Mathematics and its Applications. Oxford: Oxford University Press, 2007. Chap. 8, 134--162, MR2359346, Zbl 1286.14027.
		
		\bibitem[Kol13]{Kol13} J. Koll\'ar, \textit{Singularities of the minimal model program}, Cambridge Tracts in Math. \textbf{200} (2013), Cambridge Univ. Press. With a collaboration of S\'andor Kov\'acs, MR3057950, Zbl 1282.14028.
		
		
		\bibitem[Kol23]{Kol23} J. Koll\'{a}r, \textit{Families of varieties of general type}, Cambridge Univ. Press. (2023), MR4566297, Zbl 07658187.
		
		
		\bibitem[KM98]{KM98} J. Koll\'{a}r and S. Mori, \textit{Birational geometry of algebraic varieties}, Cambridge Tracts in Math. \textbf{134} (1998), Cambridge Univ. Press, MR1658959, Zbl 0926.14003.
		
		\bibitem[KL10]{KL10} S. J. Kov\'acs and M. Lieblich, \textit{Boundedness of families of canonically polarized manifolds: A higher dimensional analogue of Shafarevich’s conjecture}, Ann. of Math. \textbf{172} (2010), no. 3, 1719--1748, MR2726098, Zbl 1223.14040.
		
		\bibitem[KP17]{KP17} S. J. Kov\'acs and Z. Patakfalvi, \textit{Projectivity of the moduli space of stable log varieties and subadditivity of log-Kodaira dimension}, J. Amer. Math. Soc. \textbf{30} (2017), 959-1021, MR3671934, Zbl 1393.14034.
		
		\bibitem[Lai10]{Lai10} C.-J. Lai, \textit{Varieties fibered by good minimal models}, Math. Ann. \textbf{350} (2011), 533--547, MR2805635, Zbl 1221.14018.
		
		\bibitem[Laz04]{Laz04} R. Lazarsfeld, \textit{Positivity in algebraic geometry. II. Positivity for vector bundles and multiplier ideals.} Ergebnisse der Mathematik und ihrer Grenzgebiete. \textbf{3}. Folge, MR2095472, Zbl 1093.14500.
		
		
		
		
		
		
		
		\bibitem[Li24]{Li20} Z. Li, \textit{Boundedness of the base varieties of certain fibrations}, J. Lond. Math. Soc. (2) \textbf{109} (2024), no. 2, MR4704161, Zbl 07811240.
		
		
		
		
		
		
		
		
		
		
	\end{thebibliography}
\end{document}